\documentclass{article}
\pdfoutput=1

\usepackage[preprint]{neurips_2022}

\usepackage[utf8]{inputenc} 
\usepackage[T1]{fontenc}    
\usepackage[hypertexnames=false]{hyperref}       
\usepackage{url}            
\usepackage{booktabs}       
\usepackage{amsfonts}       
\usepackage{nicefrac}       
\usepackage{microtype}      
\usepackage[table]{xcolor}         

\usepackage{mathtools}
\usepackage{nicematrix}
\usepackage{makecell}
\usepackage{amsthm,amssymb}
\usepackage{algorithm}
\usepackage{algpseudocode}
\usepackage{cleveref}
\usepackage{multirow}
\usepackage{enumitem}
\allowdisplaybreaks[3]

\newtheorem{theorem}{Theorem}
\newtheorem{lemma}{Lemma}
\newtheorem{corollary}{Corollary}

\newtheorem{assumption}{Assumption}

\newtheorem{definition}{Definition}
\crefname{assumption}{assumption}{assumptions}

\DeclareMathOperator{\dom}{\mathrm{dom}}
\newcommand{\prox}{\mathrm{prox}}

\newcommand{\R}{\mathbb{R}}

\def\<#1,#2>{\langle #1,#2\rangle}

\newcommand{\Dotprod}[1]{\left\langle#1\right\rangle}

\newcommand{\norm}[1]{\|#1\|}
\newcommand{\sqn}[1]{\norm{#1}^2}

\newcommand{\vect}[1]{\begin{bmatrix*}[c]#1\end{bmatrix*}}

\newcommand{\cU}{\mathcal{U}}

\newcommand{\cL}{\mathcal{L}}

\newcommand{\cO}{\mathcal{O}}
\newcommand{\mA}{\mathbf{A}}

\DeclareMathOperator*{\argmin}{arg\,min}

\newcommand{\sX}{\R^{d_x}}
\newcommand{\sY}{\R^{d_y}}
\newcommand{\F}{\hat{F}}
\newcommand{\J}{\mathrm{J}}

\title{The First Optimal Algorithm for Smooth and Strongly-Convex-Strongly-Concave Minimax Optimization}

\author{%
	Dmitry~Kovalev\\
	KAUST\thanks{King Abdullah University of Science and Technology, Thuwal, Saudi Arabia}\\
 	\texttt{dakovalev1@gmail.com} 
 	\And
 	Alexander~Gasnikov\\
 	IITP~RAS\thanks{Institute for Information Transmission Problems RAS, Moscow, Russia}\\
 	\texttt{gasnikov@yandex.ru}
}

\begin{document}

\maketitle

\begin{abstract}
	In this paper, we revisit the smooth and strongly-convex-strongly-concave minimax optimization problem. \citet{zhang2021on} and \citet{ibrahim2020linear} established the lower bound $\Omega\left(\sqrt{\kappa_x\kappa_y} \log \frac{1}{\epsilon}\right)$ on the number of gradient evaluations required to find an $\epsilon$-accurate solution, where $\kappa_x$ and $\kappa_y$ are condition numbers for the strong convexity and strong concavity assumptions. However, the existing state-of-the-art methods do not match this lower bound: algorithms of \citet{lin2020near} and  \citet{wang2020improved} have gradient evaluation complexity $\cO\left( \sqrt{\kappa_x\kappa_y}\log^3\frac{1}{\epsilon}\right)$ and $\cO\left( \sqrt{\kappa_x\kappa_y}\log^3 (\kappa_x\kappa_y)\log\frac{1}{\epsilon}\right)$, respectively. We fix this fundamental issue by providing the first algorithm with $\cO \left(\sqrt{\kappa_x\kappa_y}\log\frac{1}{\epsilon}\right)$ gradient evaluation complexity. We design our algorithm in three steps: (i) we reformulate the original problem as a minimization problem via {\em the pointwise conjugate function}; (ii) we apply a specific variant of the proximal point algorithm to the reformulated problem; (iii) we compute the proximal operator inexactly using the optimal algorithm for operator norm reduction in monotone inclusions.
\end{abstract}

\section{Introduction}
In this paper, we revisit the smooth and strongly-convex-strongly-concave minimax optimization problem of the form
\begin{equation}\label{eq:main}
	\min_{x \in \sX}\max_{y \in \sY} r(x) + F(x,y) - g(y),
\end{equation}
where $F(x,y) \colon \sX \times \sY \rightarrow \R$ is a continuously differentiable function, $r(x) \colon \sX \rightarrow \R \cup \{+\infty\}$ and $g(y) \colon \sY \rightarrow \R \cup \{+\infty\}$ are proper lower semi-continuous convex functions.
Problem~\eqref{eq:main} has been actively studied in economics, game theory, statistics and computer science
\citep{bacsar1998dynamic,roughgarden2010algorithmic,von1947theory,facchinei2003finite,berger2013statistical}.
Recently, many applications of this problem appeared in machine learning, including adversarial training \citep{madry2017towards,sinha2017certifiable}, prediction and regression problems \citep{taskar2005structured,xu2009robustness}, reinforcement learning \citep{du2017stochastic,dai2018sbeed} and generative adversarial networks \cite{arjovsky2017wasserstein,goodfellow2014generative}.

In our paper, we focus on the case when function $f(x,y)$ is strongly convex in $x$ and strongly concave in $y$. There are several reasons to consider this function class. First, this setting is fundamental and studied by most existing works on minimax optimization.\footnote{Most existing works on minimax optimization study the convex-concave case. However, this setting can be easily reduced to the strongly-convex-strongly-concave case via the regularization technique \citep{lin2020near}.} Second, efficient algorithms initially developed for convex optimization often show state-of-the-art performance in non-convex applications  \citep{kingma2014adam,reddi2019convergence,duchi2011adaptive}. 
Finally, we will further see that this fundamental setting is utterly understudied and lacks answers to even the most basic questions such as ``What is the best possible algorithm for solving a problem in this setting?''\footnote{In contrast to smooth convex-concave minimax optimization, the answer to this question for smooth convex minimization was given by \citet{nesterov1983method} several decades ago.}

\subsection{Related Work}
Until recently, the best-known gradient evaluation complexity of solving problem~\eqref{eq:main} was $\cO \left(\max\left\{\kappa_x,\kappa_y\right\}\log\frac{1}{\epsilon}\right)$ \citep{tseng2000modified,nesterov2006solving,gidel2018variational}, where $\kappa_x$ and $\kappa_y$ denote the condition numbers of functions $f(\cdot,y)$ and $f(x,\cdot)$, respectively. The first attempt to provide an algorithm with an ``accelerated'' convergence rate was the work of \citet{alkousa2019accelerated}. They provided an algorithm with $\cO\left(\min\left\{\kappa_x\sqrt{\kappa_y},\kappa_y\sqrt{\kappa_x}\right\} \log^2\frac{1}{\epsilon}\right)$ gradient evaluation complexity. This result was subsequently improved up to $\cO\left( \sqrt{\kappa_x\kappa_y}\log^3\frac{1}{\epsilon}\right)$ by \citet{lin2020near} and $\cO\left( \sqrt{\kappa_x\kappa_y}\log^3 (\kappa_x\kappa_y)\log\frac{1}{\epsilon}\right)$ by \citet{wang2020improved}. However, these results do not match the lower complexity bound $\Omega \left(\sqrt{\kappa_x\kappa_y}\log\frac{1}{\epsilon}\right)$ established by \citet{zhang2021on,ibrahim2020linear}. Hence, we have the following fundamental open problem:
\begin{center}
	\em
	Can we design an algorithm that achieves the lower gradient evaluation complexity bound in smooth and strongly-convex-strongly-concave minimax optimization?
\end{center}

It is worth mentioning that this open question was answered positively in the work of \citet{kovalev2021accelerated} in the case of minimax problems with bilinear coupling, i.e., when $F(x,y)= p(x) + x^\top \mA y - q(y)$, where $p(x)$ and $q(y)$ are smooth and strongly convex functions, and $\mA$ is a $d_x\times d_y$ matrix. However, the algorithm provided in this work does not apply to the general minimax problem~\eqref{eq:main}.

\begin{table}
	\caption{Comparison of the state-of-the-art algorithms for solving smooth and strongly-convex-strongly-concave minimax problems in the number of gradient evaluations required to find an $\epsilon$-accurate solution (\Cref{def:accuracy}).}
	\label{tab:main}
	\centering
	\begin{NiceTabular}{|c|c|}[code-before = \rowcolor{gray!25}{8}]
		\toprule
		\bf Reference & \bf Gradient Complexity\\
		\midrule
		\makecell{\citet{tseng2000modified}} & \Block{3-1}{$\cO \left(\max\left\{\kappa_x,\kappa_y\right\}\log\frac{1}{\epsilon}\right)$}\\
		\cmidrule{1-1}
		\makecell{ \citet{nesterov2006solving}}&\\
		\cmidrule{1-1}
		\makecell{\citet{gidel2018variational}} & \\
		\midrule
		\citet{alkousa2019accelerated} & $\cO\left(\min\left\{\kappa_x\sqrt{\kappa_y},\kappa_y\sqrt{\kappa_x}\right\} \log^2\frac{1}{\epsilon}\right)$ \\
		\midrule
		\makecell{\citet{lin2020near}}& $\cO\left( \sqrt{\kappa_x\kappa_y}\log^3\frac{1}{\epsilon}\right)$\\
		\midrule
		\makecell{\citet{wang2020improved}}&$\cO\left( \sqrt{\kappa_x\kappa_y}\log^3 (\kappa_x\kappa_y)\log\frac{1}{\epsilon}\right)$\\
		\midrule
		\bf \Cref{alg:foam} (This paper)&$\cO \left(\sqrt{\kappa_x\kappa_y}\log\frac{1}{\epsilon}\right)$\\
		\midrule
		\makecell{Lower Bound \citep{zhang2021on,ibrahim2020linear}}&$\Omega \left(\sqrt{\kappa_x\kappa_y}\log\frac{1}{\epsilon}\right)$\\
		\bottomrule
	\end{NiceTabular}
\end{table}

\subsection{Main Contributions}

We develop the first optimal algorithm for solving problem~\eqref{eq:main} in the smooth and strongly-convex-strongly-concave regime, which is the main contribution of this work. We split the algorithm development in three steps:
\begin{enumerate}[topsep=0em]
	\item[\bf (i)]  In \Cref{sec:min}, we reformulate problem~\eqref{eq:main} as a particular minimization problem.
	\item[\bf (ii)] In \Cref{sec:pp}, we develop a specific variant of the accelerated proximal point algorithm (\Cref{alg:pp}) which will be used as a baseline for the optimal algorithm construction.
	\item[\bf (ii)] In \Cref{sec:inclusion}, we develop an optimal algorithm for operator norm reduction in monotone inclusion problems, which will be used for the proximal operator computation in \Cref{alg:pp}.
\end{enumerate}
In the final \Cref{sec:foam}, we summarize these three steps by describing the optimal algorithm construction and showing that the complexity of the proposed algorithm matches the lower bound.

As mentioned before, in \Cref{sec:inclusion}, we develop an optimal algorithm for operator norm reduction in composite monotone inclusion problems of the form~\eqref{eq:inclusion}, which is the second main contribution of this work. To the best of our knowledge, there is only one optimal algorithm of \citet{yoon2021accelerated}, which works for Lipschitz-continuous operators only, i.e., when $B(u) \equiv 0$ in problem~\eqref{eq:inclusion}. In contrast to this, our algorithm works in the composite case with general maximally monotone operator $B(u)$.

\section{Preliminaries}

The following assumptions formalize the smoothness, strong convexity, and strong concavity properties of function $f(x,y)$.

\begin{assumption}\label{ass:convexity}
	Function $F(x,y)$ is $\mu_x$-strongly convex in $x$, where $\mu_x > 0$. That is, the following inequality holds for all $x_1,x_2 \in \sX, y \in \sY$:
	\begin{equation}
		F(x_2,y) \geq F(x_1,y) + \<\nabla_x F(x_1,y), x_2 - x_1> + (\mu_x/2)\sqn{x_2 - x_1}.
	\end{equation}
\end{assumption}

\begin{assumption}\label{ass:concavity}
	Function $F(x,y)$ is $\mu_y$-strongly concave in $y$, where $\mu_y > 0$. That is, the following inequality holds for all $x \in \sX, y_1,y_2 \in \sY$:
	\begin{equation}
		F(x,y_2) \leq F(x,y_1) + \<\nabla_y F(x,y_1), y_2 - y_1> - (\mu_y/2)\sqn{y_2 - y_1}.
	\end{equation}
\end{assumption}

\begin{assumption}\label{ass:smoothness}
	Function $F(x,y)$ is $L$-smooth. That is, the following inequality holds for all $x_1,x_2 \in \sX, y_1,y_2 \in \sY$:
	\begin{equation}
		\sqn{\nabla F(x_1,y_1) - \nabla F(x_2,y_2)} \leq L^2 \left(\sqn{x_1 - x_2} + \sqn{y_1 - y_2}\right).
	\end{equation}
\end{assumption}

Under these assumptions, by $\kappa_x=\frac{L}{\mu_x}$ and $\kappa_y=\frac{L}{\mu_x}$, we denote the condition numbers of functions $F(\cdot,y)$ and $F(x, \cdot)$, respectively. The following assumption formalizes the properties of regularizers $r(x)$ and $g(y)$.

\begin{assumption}
	Functions $r(x)$ and $g(y)$ are convex, lower semi-continuous and proper, i.e., there exist $\bar{x} \in \sX, \bar{y} \in \sY$ such that $r(\bar{x}), g(\bar{y}) < +\infty$.
\end{assumption}

By $(x^*,y^*) \in \sX\times \sY$, we denote the solution of problem~\eqref{eq:main}, which is characterized via the first-order optimality conditions
\begin{equation}\label{eq:opt}
	\begin{cases}
			\begin{split}
				-\nabla_x F(x^*,y^*) &\in \partial r(x^*),\\
				\nabla_y F(x^*,y^*) &\in \partial g(y^*).
				\end{split}
		\end{cases}
\end{equation}
Note that there exists a unique solution to the problem due to the strong convexity and strong concavity assumptions (\Cref{ass:convexity,ass:concavity}). Hence, for any point $(x,y) \in  R^{d_x}\times \sY$, we can use squared distance to the solution $\sqn{ x - x^*} + \sqn{y - y^*}$ as an optimality criterion. We formalize it through the following definition.
\begin{definition}\label{def:accuracy}
	We call a pair of vectors $(x,y) \in \sX\times \sY$  an $\epsilon$-accurate solution of problem~\eqref{eq:main} for a given accuracy $\epsilon > 0$ if it satisfies
	\begin{equation}
		\sqn{ x - x^*} + \sqn{y - y^*} \leq \epsilon.
	\end{equation}
\end{definition}

\section{Step I: Reformulation via Pointwise Conjugate Function}\label{sec:min}

In this section, we reformulate problem~\eqref{eq:main} as a particular convex minimization problem. This reformulation will be beneficial because minimization problems are typically easier to solve than minimax optimization problems. 

\subsection{Pointwise Conjugate Function}
We start by introducing {\em the pointwise conjugate function} which will be the main component of our problem reformulation.
Let function $\F(x,y)\colon \sX \times \sY \rightarrow \R$ be defined as
\begin{equation}\label{eq:Fhat}
	\F(x,y) = F(x,y) - \frac{\mu_x}{2}\sqn{x} + \frac{\mu_y}{2}\sqn{y}.
\end{equation}
One can observe that function $\F(x,y)$ is smooth, convex in $x$, and concave in $y$ due to \Cref{ass:convexity,ass:concavity,ass:smoothness}.
Now, the pointwise conjugate function $G(z,y) \colon \sX\times \sY \rightarrow \R$ is defined as follows:
\begin{equation}
	G(z,y) = \sup_{x\in \sX} \left[\<x,z> - r(x) - \F(x,y) + g(y) \right].
\end{equation}
One can observe that for fixed $y \in \sY$, function $G(\cdot, y)$ is nothing else but the Fenchel conjugate\footnote{Recall that for a convex function $h(x)$, Fenchel conjugate is defined as $h^*(z) = \sup_x [\<z,x> - h(x)]$.} of function $r(\cdot) + \F(\cdot,y) - g(y)$.
Moreover, function $G(z,y)$ is defined as a pointwise supremum of a family of convex and lower semi-continuous functions $\left\{\varphi_x(z,y) = \<x,z> - r(x) - F(x,y) + g(y)\mid x \in \sX\right\}$. Hence, $G(z,y)$ is also convex and lower semi-continuous function. The following lemma provides a characterization of the subdifferential of the pointwise conjugate function.
\begin{lemma}\label{lem:subdifferential}
	Let $z,x \in \sX$ and $ y, w \in \sY$ be arbitrary vectors that satisfy
	\begin{equation}
		z - \nabla_x \F(x,y) \in \partial r(x), \; w + \nabla_y \F(x,y) \in \partial g(y).
	\end{equation}
	Then, $G(z,y) = \<z,x> - r(x) - \F(x,y) + g(y)$ and $(x,w) \in \partial G(z,y)$.	
\end{lemma}

\subsection{Reformulation of the Minimax Optimization Problem}

Now, we introduce the following minimization problem:
\begin{equation}\label{eq:min}
	\min_{z \in \sX, y\in \sY} \left[P(z,y) = \frac{\mu_x^{-1}}{2}\sqn{z} + \frac{\mu_y}{2}\sqn{y} + G(z,y) \right]
\end{equation}
It turns out that this minimization problem can be seen as a reformulation of problem~\eqref{eq:main}. This is justified by the following lemma.
\begin{lemma}\label{lem:solution}
	Problem~\eqref{eq:min} has a unique solution $(z^*,y^*) \in \sX \times \sY$, where
	\begin{equation}\label{eq:zstar}
		z^* = -\mu_x x^*
	\end{equation}
	and $(x^*,y^*)$ is the unique solution of problem~\eqref{eq:main}.
\end{lemma}
\Cref{lem:solution} implies that if we find an approximate solution $(z, y) \in \sX\times \sY$ to problem~\eqref{eq:min}, a pair of vectors $(-\mu_x^{-1} z, y) \in \sX\times \sY$ will be an approximate solution to the original minimax problem. 

The idea of reformulating the minimax optimization problem as a minimization problem is not new and has been used in the state-of-the-art works of \citet{lin2020near,wang2020improved,alkousa2019accelerated}. However, their reformulation is different from ours and has several disadvantages. In particular, it does not allow for building the optimal algorithm for solving problem~\eqref{eq:main}. We provide a detailed discussion of this in the Appendix.

\section{Step II: Accelerated Proximal Point Method}\label{sec:pp}

\begin{algorithm}[t]
	\caption{Accelerated Gradient Method}
	\label{alg:nesterov}
	\begin{algorithmic}[1]
		\State {\bf input:} $z^0 = z_f^0 \in \sX, y^0 = y_f^0\in \sY$
		\State {\bf parameters:} $\alpha \in (0,1]$, $\eta_z,\eta_y,\theta_z, \theta_y > 0$, $K \in \{1,2,\ldots\}$
		\For{$k = 0,1,2,\ldots,K-1$}
		\State$(z_g^k,y_g^k) = \alpha (z^k,y^k) + (1-\alpha) (z_f^k, y_f^k)$
		\State $z_f^{k+1} = z_g^k - \theta_z \nabla_z P(z_g^k,y_g^k)$ \label{nesterov:line:z:1}
		\State $y_f^{k+1} = y_g^k - \theta_y\nabla_y P(z_g^k,y_g^k)$ \label{nesterov:line:y:1}
		\State$z^{k+1} = z^k + \eta_z\mu_z (z_g^k - z^k) + \eta_z\theta_z^{-1}(z_f^{k+1} - z_g^k)$ \label{nesterov:line:z:2}
		\State $y^{k+1} = y^k + \eta_y\mu_y (y_g^k - y^k) + \eta_y\theta_y^{-1}(y_f^{k+1} - y_g^k)$\label{nesterov:line:y:2}
		\EndFor
		\State {\bf output:} $(z^K, y^K)$
	\end{algorithmic}
\end{algorithm}

\begin{algorithm}[t]
	\caption{Accelerated Proximal Point Algorithm}
	\label{alg:pp}
	\begin{algorithmic}[1]
		\State {\bf input:} $z^0 = z_f^0 \in \sX, y^0 = y_f^0\in \sY$
		\State {\bf parameters:} $\alpha \in (0,1]$, $\eta_z,\eta_y,\theta_y > 0$, $K \in \{1,2,\ldots\}$
		\For{$k = 0,1,2,\ldots,K-1$}
		\State $(z_g^k,y_g^k) = \alpha (z^k,y^k) + (1-\alpha) (z_f^k, y_f^k)$ \label{pp:line:convex}
		\State Find $(x_f^{k+1}, y_f^{k+1}, z_f^{k+1},w_f^{k+1}) \in \sX\times\sY\times \sX \times \sY$ that satisfy \eqref{pp:monteiro-svaiter} \label{pp:line:aux}
		\State $z^{k+1} = z^k + \eta_z \mu_x^{-1} (z_f^{k+1} - z^k) - \eta_z (x_f^{k+1} + \mu_x^{-1} z_f^{k+1})$\label{pp:line:z} 
		\State $y^{k+1} = y^k + \eta_y \mu_y (y_f^{k+1} - y^k) - \eta_y (w_f^{k+1} + \mu_y y_f^{k+1})$ \label{pp:line:y} 
		\EndFor
		\State {\bf output:} $(z^K, y^K)$
	\end{algorithmic}
\end{algorithm}

In this section, we develop the main algorithmic framework for solving problem~\eqref{eq:min}, which is formalized as \Cref{alg:pp}. We give the intuition behind the development of \Cref{alg:pp} and provide its theoretical analysis. Further, in \Cref{sec:foam}, we will use this algorithmic framework to develop the first optimal algorithm for solving main problem~\eqref{eq:main}.

\subsection{Nesterov Acceleration}

It is well-known that Accelerated Gradient Method of \citet{nesterov1983method,nesterov2003introductory} is the optimal algorithm for solving smooth (strongly-)convex minimization problems. Therefore, we could try to apply this method to solving problem~\eqref{eq:min}, which is formalized as \Cref{alg:nesterov}.
Note that we used the notation $\mu_z = \mu_x^{-1}$ in \Cref{alg:nesterov}, which is the strong convexity parameter of $P(z,y)$ in $z$. Unfortunately, function $P(z,y)$ can be non-smooth,
and the gradient $\nabla P(z_g^k,y_g^k)$ can be undefined. It means that \Cref{alg:nesterov} cannot be applied to problem~\eqref{eq:min}.

\subsection{Moreau-Yosida Regularization}

In order to avoid the issues caused by the non-smoothness of function $P(z,y)$, we use the Moreau-Yosida regularization \citep{moreau1962fonctions,yosida2012functional}. Consider a function $P^{\theta_z,\theta_y}(z,y)$ defined in the following way:
\begin{equation}
	P^{\theta_z,\theta_y}(z,y) = \min_{z^+ \in \sX, y^+ \in \sY} \frac{1}{2\theta_z}\sqn{z^+ - z} + \frac{1}{2\theta_y}\sqn{y^+ - y} + P(z^+,y^+),
\end{equation}
where $\theta_z,\theta_y > 0$. Function $P^{\theta_z,\theta_y}(z,y)$ is called the Moreau envelope of function $P(z,y)$. The Moreau envelope has two crucial properties. First, it is a smooth function. Second, it has the same minimizers as function $P(z,y)$:
\begin{equation}\label{eq:moreau_min}
	(z^*,y^*) = \argmin_{z\in\sX, y\in\sY} P^{\theta_z,\theta_y}(z,y).
\end{equation}
The latter means that we could apply Accelerated Gradient Method to problem~\eqref{eq:moreau_min}, which would give us an efficient algorithm for solving problem~\eqref{eq:min}. Further, we are going to construct such an algorithm.

\subsection{Construction of the Algorithm}

We start the construction of our algorithm by computing the gradient $\nabla P^{\theta_z,\theta_y}(z_g^k,y_g^k)$. 
The theory of the Moreau-Yosida regularization \citep{lemarechal1997practical} suggests that the gradient of the Moreau envelope can be computed in the following way:
\begin{equation}
	\nabla P^{\theta_z,\theta_y}(z_g^k,y_g^k) = \vect{\theta_z^{-1}(z_g^k - z_f^{k+1})\\ \theta_y^{-1}(y_g^k - y_f^{k+1})},
\end{equation}
where $(z_f^{k+1},y_f^{k+1}) \in \sX\times \sY$ is computed via the following auxiliary minimization problem:
\begin{equation}\label{eq:aux}
	(z_f^{k+1}, y_f^{k+1}) = \argmin_{z \in \sX, y \in \sY} \frac{1}{2\theta_z}\sqn{z - z_g^k} + \frac{1}{2\theta_y}\sqn{y - y_g^k} + P(z,y).
\end{equation}

Further, we choose parameter $\theta_z = \mu_z^{-1} = \mu_x$ and write the first-order optimality conditions for this problem using the definition of function $P(z,y)$:
\begin{equation}\label{eq:8}
	\vect{\mu_x^{-1}(z_f^{k+1} - z_g^k) +  \mu_x^{-1} z_f^{k+1}\\\theta_y^{-1}(y_f^{k+1} - y_g^k) + \mu_y y_f^{k+1} } \in-\partial G(z_f^{k+1}, y_f^{k+1}).
\end{equation}
The latter condition involves the subdifferential $\partial G(z,y)$. Hence, we can rewrite this condition using \Cref{lem:subdifferential}, which provides the characterization of $\partial G(z,y)$\footnote{To be precise, \Cref{lem:subdifferential} implies the relation \eqref{eq:9} $\Rightarrow$ \eqref{eq:8} rather than the equivalence \eqref{eq:9} $\Leftrightarrow$ \eqref{eq:8}. However, this is not an issue because we provide the intuition behind the algorithm development in this section. The rigorous proofs are postponed to the Appendix.}:
\begin{equation}\label{eq:9}
	\begin{alignedat}{2}
		z_f^{k+1} - \nabla_x \F(x_f^{k+1}, y_f^{k+1}) &\in \partial r(x_f^{k+1}),&\qquad
		x_f^{k+1} + \mu_x^{-1}(z_f^{k+1} - z_g^k) +  \mu_x^{-1} z_f^{k+1} &= 0,\\
		w_f^{k+1} + \nabla_y \F(x_f^{k+1}, y_f^{k+1}) &\in \partial g(y_f^{k+1}),&
		w_f^{k+1} + \theta_y^{-1}(y_f^{k+1} - y_g^k) + \mu_y y_f^{k+1}  &= 0.
	\end{alignedat}
\end{equation}
where $x_f^{k+1} \in \sX$ and $w_f^{k+1}\in \sY$  are auxiliary vectors. From \eqref{eq:9} we get  
\begin{align*}
	\mu_x^{-1}(z_f^{k+1} - z_g^k) &= - (x_f^{k+1} + \mu_x^{-1} z_f^{k+1}),\\
	\theta_y^{-1}(y_f^{k+1} - y_g^k) &= -(w_f^{k+1} + \mu_y y_f^{k+1}),
\end{align*}
which we plug into \cref{nesterov:line:z:2,nesterov:line:y:2} of \Cref{alg:nesterov}.

Finally, we replace the computation of $(z_f^{k+1} ,y_f^{k+1})$ on \cref{nesterov:line:z:1,nesterov:line:y:1} of \Cref{alg:nesterov} using condition~\eqref{eq:9}. It turns out that we can use the following relaxed version of condition~\eqref{eq:9} without hurting the convergence properties of the resulting algorithm:
\begin{equation}\label{pp:monteiro-svaiter}
	\begin{cases}
		\begin{split}
				&z_f^{k+1} - \nabla_x \F(x_f^{k+1}, y_f^{k+1}) \in \partial r(x_f^{k+1}),\\
				&w_f^{k+1} + \nabla_y \F(x_f^{k+1}, y_f^{k+1}) \in \partial g(y_f^{k+1}),\\
				&\frac{8}{\mu_x}\sqn{\Delta_x^k}+\theta_y\sqn{\Delta_y^k}
				\leq
				\frac{\mu_x}{8}\sqn{x_f^{k+1} +\mu_x^{-1} z_g^k}+\theta_y^{-1}\sqn{y_f^{k+1} - y_g^k},\\
			\end{split}
	\end{cases}
\end{equation}
where $\Delta_x^k$ and $\Delta_y^k$ are defined as follows:
\begin{equation}
	\begin{cases}
			\begin{split}
					\Delta_x^k &= z_f^{k+1} + \frac{\mu_x}{2}(x_f^{k+1} - \mu_x^{-1} z_g^k),\\
					\Delta_y^k &= w_f^{k+1} + \mu_y y_f^{k+1} + \theta_y^{-1} (y_f^{k+1} - y_g^k).
				\end{split}
		\end{cases}
\end{equation}

\subsection{Convergence of the Algorithm}

After applying all the modifications mentioned above to \Cref{alg:nesterov}, we obtain \Cref{alg:pp}.
\Cref{pp:thm} provides the iteration complexity of \Cref{alg:pp}. The proof of \Cref{pp:thm} can be found in the Appendix.

\begin{theorem}\label{pp:thm}
	Let $\eta_z, \eta_y$ be defined as
	\begin{equation}\label{pp:eta}
		\eta_z = \mu_x/2, \quad \eta_y = \min\left\{{1}/({2\mu_y}), {\theta_y}/({2\alpha})\right\}.
	\end{equation}
	Then, to find an $\epsilon$-accurate solution of problem~\eqref{eq:main}, \Cref{alg:pp} requires the following number of iterations:
	\begin{equation}
		K = \cO\left(\max\left\{\frac{1}{\alpha}, \frac{\alpha}{\theta_y\mu_y}\right\}\log\frac{1}{\epsilon}\right).
	\end{equation}
	In this case, the $\epsilon$-accurate solution will be given as $(-\mu_x^{-1}z^K, y^K)$, where $(z^K,y^K)$ is the output of \Cref{alg:pp}.

\end{theorem}

Unfortunately, \Cref{alg:pp} cannot be applied to solving problem~\eqref{eq:main} in its current form because it requires finding vectors $(x_f^{k+1},y_f^{k+1},z_f^{k+1},w_f^{k+1})$ that satisfy condition \eqref{pp:monteiro-svaiter} on \cref{pp:line:aux} at each iteration. Further, we will show that finding these vectors can be seen as finding an approximate solution to a particular monotone inclusion problem. In \Cref{sec:inclusion}, we will provide an optimal algorithm for solving such monotone inclusions. In \Cref{sec:foam}, we will show how to combine this algorithm with \Cref{alg:pp} and obtain the first optimal algorithm for solving main problem~\eqref{eq:main}.

\section{Step III: Operator Norm Reduction in Monotone Inclusions}\label{sec:inclusion}

\begin{algorithm}[t]
	\caption{Extra Anchored Gradient for Monotone Inclusions}
	\label{alg:eag}
	\begin{algorithmic}[1]
		\State {\bf input:} $u^{-1} \in \R^d$
		\State {\bf parameters:} $\lambda > 0$, $T \in \{1,2,\ldots\}, \{\beta_t\}_{t=0}^{T-1} \subset (0,1)$
		\State $u^0 = \J_{\lambda B} (u^{-1} - \lambda A(u^{-1}))$\label{eag:line:0}
		\State $a^0 = A(u^0)$
		\State$b^0 = \frac{1}{\lambda}(u^{-1} - \lambda A(u^{-1}) - u^0)$\label{eag:line:3} \algorithmiccomment{$b^0 \in B(u^0)$}
		\For{$t = 0,1,2\ldots, T-1$}
		\State$u^{t+1/2} = u^t + \beta_t(u^0 - u^t)- \lambda (a^t + b^t)$\label{eag:line:1} 
		\State  $u^{t+1} = \J_{\lambda B}(u^t + \beta_t(u^0 - u^t) - \lambda A (u^{t+1/2}))$
		\State $a^{t+1} = A(u^{t+1})$
		\State$b^{t+1} = \frac{1}{\lambda}(u^t + \beta_t(u^0 - u^t) - \lambda A (u^{t+1/2}) - u^{t+1})$\label{eag:line:2} \algorithmiccomment{$b^{t+1} \in B(u^{t+1})$}
		\EndFor
		\State {\bf output:} $(u^T, a^T + b^T)$ \algorithmiccomment{$b^T \in B(u^T)$}
	\end{algorithmic}
\end{algorithm}

In this section, we consider the following monotone inclusion problem:
\begin{equation}\label{eq:inclusion}
	\text{find } u^* \in \R^d \quad \text{such that}\quad 0 \in A(u^*) + B(u^*),
\end{equation}
where $A(u),B(u)\colon \R^d\rightrightarrows \R^d$  are maximally monotone mappings. We are interested in the case when $A(u)$ is single-valued and Lipschitz continuous. The properties of operators $A(u)$ and $B(u)$ are formalized through the following assumptions.
\begin{assumption}\label{ass:A}
	Mapping $A(u)\colon \R^d \rightarrow \R^d$ is single-valued, $M$-Lipschitz and monotone. That is, for all $u_1,u_2 \in \R^d$,
	$
		\<A(u_1) - A(u_2), u_1 - u_2> \geq 0
	$
	and
	$
		\norm{A(u_1) - A(u_2)} \leq M\norm{u_1 - u_2}.
	$
\end{assumption}
\begin{assumption}\label{ass:B}
	Mapping $B(u)\colon \R^d \rightrightarrows \R^d$ is maximally monotone and possibly multivalued. That is, mapping $B(u)$ satisfies the following conditions:
	\begin{enumerate}[topsep=0em]
		\item $B(u)$ is monotone, i.e., for all $u_1,u_2 \in \dom B$, $b_1 \in B(u_1)$, $b_2 \in B(u_2)$ the following inequality holds:
		$
			\<u_1 - u_2, b_1 - b_2> \geq 0,
		$
		where $\dom B = \{u \in \R^d \mid B(u) \neq \emptyset\}$.
		\item The graph
		$
			\mathrm{gph} B = \{ (u,b) \in \R^d \times \R^d \mid b \in B(u)\}
		$
		is not properly contained in the graph of any other monotone mapping on $\R^d$.
	\end{enumerate}
\end{assumption}

Note that mapping $A(u)$ is also maximally monotone because it is monotone and continuous \citep[Example 12.7]{rockafellar2009variational}. 
Further, we will use an operator $\J_{\lambda B}(u) \colon \R^d \rightrightarrows \R^d$ which is defined as 
\begin{equation}
	u^+ \in \J_{\lambda B}(u)  \quad \text{if and only if}\quad \lambda^{-1}(u - u^+) \in  B(u^+),
\end{equation}
where $\lambda > 0$.
This mapping is called the resolvent of mapping $B(u)$.
The maximal monotonicity of $B(u)$ implies that the resolvent $\J_{\lambda B}(u)$ is single-valued for all $u\in \R^d$ \citep[Theorem 12.12]{rockafellar2009variational}.

Now, we are ready to present \Cref{alg:eag} for solving monotone inclusion problem \eqref{eq:inclusion}. The design of our algorithm is based on the Extra Anchored Gradient Algorithm of \citet{yoon2021accelerated}. The critical difference between the algorithm of \citet{yoon2021accelerated} and \Cref{alg:eag} is that the algorithm of \citet{yoon2021accelerated} can be applied to problem~\eqref{eq:inclusion} in the case $B(u) \equiv \{0\}$ only.
Therefore, our \Cref{alg:eag} can be seen as an extension of the algorithm of \citet{yoon2021accelerated} for general monotone inclusion problems of the form~\eqref{eq:inclusion}.

The following theorem provides the convergence guarantees for \Cref{alg:eag}. The proof of the theorem can be found in the Appendix.
\begin{theorem}\label{eag:thm}
	Assume that there exists at least a single solution $u^*$ to problem~\eqref{eq:inclusion}. Let $\beta_t$ be defined as follows
	\begin{equation}\label{eag:beta}
		\beta_t ={2}/(t + 3).
	\end{equation}
	Let $\lambda$ be defined as
	\begin{equation}\label{eag:lambda}
		\lambda = {1}/({\sqrt{5}M}).
	\end{equation}
	Then, the following inequality holds
	\begin{equation}
		\sqn{a^T+b^T} \leq
		\frac{288M^2}{(T+1)^2}\sqn{u^{-1} - u^*},
	\end{equation}
	where $a^T = A(u^T)$ and $b^T \in B(u^T)$.
\end{theorem}

\section{Final Step: The First Optimal Algorithm for Minimax Optimization}\label{sec:foam}

In this section, we construct the first optimal algorithm for solving main problem~\eqref{eq:main}. In order to do this,
we use \Cref{alg:eag} to compute vectors $(x_f^{k+1},y_f^{k+1},z_f^{k+1},w_f^{k+1})$ on \cref{pp:line:aux} of \Cref{alg:pp}. Further, we describe the construction of our algorithm in detail.

\begin{algorithm}[t]
	\caption{FOAM: The First Optimal Algorithm for Minimax Optimization}
	\label{alg:foam}
	\begin{algorithmic}[1]
		\State {\bf input:} $z^0 = z_f^0 \in \sX$, $y^0 = y_f^0 \in \sY$
		\State {\bf parameters:} $\alpha \in (0,1]$, $\eta_z,\eta_y,\theta_y> 0$, $\{\beta_t\}_{t=0}^{\infty} \subset (0,1)$, $\lambda,\gamma_x,\gamma_y > 0$, $K \in \{1,2,\ldots\}$
		\For{$k=0,1,2,\ldots,K-1$}
		\State $(z_g^k,y_g^k) = \alpha (z^k,y^k) + (1-\alpha) (z_f^k, y_f^k)$
		\State $(x^{k,-1}, y^{k,-1}) = (-\mu_x^{-1} z_g^k, y_g^k)$ \label{foam:line:xyi}
		\State $x^{k,0} = \prox_{\gamma_x\lambda r(\cdot)}(x^{k,-1} - \gamma_x \lambda a_x^k(x^{k,-1}, y^{k,-1}))$
		\State $y^{k,0} = \prox_{\gamma_y \lambda g(\cdot)}(y^{k,-1} - \gamma_y \lambda a_y^k(x^{k,-1}, y^{k,-1}))$
		\State $b_x^{k,0} = \frac{1}{\gamma_x\lambda}(x^{k,-1} - \gamma_x \lambda a_x^k(x^{k,-1}, y^{k,-1}) - x^{k,0})$
		\State $b_y^{k,0} = \frac{1}{\gamma_y\lambda}(y^{k,-1} - \gamma_y \lambda a_y^k(x^{k,-1},y^{k,-1}) - y^{k,0})$
		\State $t = 0$
		\While{condition \eqref{foam:while} is not satisfied}\label{foam:line:while}
		\State $x^{k,t+1/2} = x^{k,t} + \beta_t(x^{k,0} - x^{k,t})- \gamma_x\lambda (a_x^k(x^{k,t}, y^{k,t}) +b_x^{k,t})$
		\State $y^{k,t+1/2} = y^{k,t}+ \beta_t(y^{k,0} - y^{k,t}) - \gamma_y\lambda (a_y^k(x^{k,t}, y^{k,t}) +b_y^{k,t})$
		\State $x^{k,t+1} = \prox_{\gamma_x\lambda r(\cdot)}(x^{k,t}+\beta_t(x^{k,0} - x^{k,t}) - \gamma_x \lambda a_x^k(x^{k,t+1/2}, y^{k,t+1/2}))$
		\State $y^{k,t+1} = \prox_{\gamma_y \lambda g(\cdot)}(y^{k,t} +\beta_t(y^{k,0} - y^{k,t})- \gamma_y \lambda a_y^k(x^{k,t+1/2}, y^{k,t+1/2}))$
		\State $b_x^{k,t+1} = \frac{1}{\gamma_x\lambda}(x^{k,t} +\beta_t(x^{k,0} - x^{k,t})- \gamma_x \lambda a_x^k(x^{k,t+1/2}, y^{k,t+1/2}) - x^{k,t+1})$
		\State $b_y^{k,t+1} = \frac{1}{\gamma_y\lambda}(y^{k,t}+\beta_t(y^{k,0} - y^{k,t}) - \gamma_y \lambda a_y^k(x^{k,t+1/2},y^{k,t+1/2}) - y^{k,t+1})$
		\State $t = t+1$
		\EndWhile
		\State $t^k = t$ \label{foam:line:t}
		\State $(x_f^{k+1},y_f^{k+1}) = (x^{k,t^k},y^{k,t^k})$ \label{foam:line:xyf}
		\State $(z_f^{k+1},w_f^{k+1}) = ( \nabla_x \F(x_f^{k+1}, y_f^{k+1}) + b_x^{k,t^k}, -\nabla_y \F(x_f^{k+1}, y_f^{k+1}) + b_y^{k,t^k})$ \label{foam:line:zwf}
		\State $z^{k+1} = z^k + \eta_z \mu_x^{-1} (z_f^{k+1} - z^k) - \eta_z (x_f^{k+1} + \mu_x^{-1} z_f^{k+1})$
		\State $y^{k+1} = y^k + \eta_y \mu_y (y_f^{k+1} - y^k) - \eta_y (w_f^{k+1} + \mu_y y_f^{k+1})$
		\EndFor
		\State {\bf output:} $(-\mu_x^{-1}z^K, y^K)$
	\end{algorithmic}
\end{algorithm}

\subsection{Construction of the Algorithm}

As mentioned in \Cref{sec:pp}, \Cref{alg:pp} cannot be applied to solving problem~\eqref{eq:main} in its current form because it requires finding the vectors satisfying condition~\eqref{pp:monteiro-svaiter} on \cref{pp:line:aux} at each iteration. Further, we will show how to do this using \Cref{alg:eag}. Let $\R^d = \sX \times \sY$. For each $k\in\{0,1,2,\ldots\}$ consider operators $A^k(u)\colon \R^d \rightarrow \R^d$ and $B(u)\colon \R^d \rightrightarrows \R^d$ defined as follows:
\begin{equation}\label{eq:AB}
	\begin{split}
		A^k(u) = \vect{\sqrt{\gamma_x} a_x^k(x,y)\\
			 \sqrt{\gamma_y}a_y^k(x,y)}, \quad
		B(u) =  \left\{\vect{\sqrt{\gamma_x}b_x\\\sqrt{\gamma_y}b_y}\; \middle\vert \; b_x \in \partial r(x), b_y \in \partial g(y) \right\},
	\end{split}
\end{equation}
where $\gamma_x,\gamma_y >0$ are parameters, variable $u \in \R^d$ is defined as 
\begin{equation}
	u=(\gamma_x^{-1/2}x, \gamma_y^{-1/2}y), \quad \text{where} \quad (x,y) \in \sX \times \sY,
\end{equation}
and operators $a_x^k(x,y)\colon \R^d \rightarrow \sX$ and $a_y^k(x,y) \colon \R^d \rightarrow \sY$ are defined as
\begin{equation}
	\begin{split}
		a_x^k(x,y) &= \nabla_x \F(x,y) + \frac{\mu_x}{2}(x- \mu_x^{-1}z_g^k),\\
		a_y^k(x,y) &= -\nabla_y \F(x,y) + \mu_y y + \theta_y^{-1}(y - y_g^k).
	\end{split}
\end{equation}
One can observe that operators $A^k(u)$ and $B(u)$ satisfy \Cref{ass:A,ass:B}. This is justified by the following lemma.
\begin{lemma}\label{lem:M}
	Operator $A^k(u)$, defined by \eqref{eq:AB}, is monotone and $M$-Lipschitz, where $M$ is given as
	\begin{equation}\label{foam:M}
		M = 2\max\{\gamma_x L,\gamma_y(L+\theta_y^{-1})\}.
	\end{equation}
	Operator $B(u)$, defined by \eqref{eq:AB}, is maximally monotone.
\end{lemma}

Now, we are ready to construct the first optimal algorithm for solving main problem~\eqref{eq:main} which is formalized as \Cref{alg:foam}.
In order to do this, we use \Cref{alg:eag} to perform the computations on \cref{pp:line:aux} of \Cref{alg:pp}. Consider the $k$-th iteration of \Cref{alg:pp} and replace \cref{pp:line:aux} of \Cref{alg:pp} with the lines of \Cref{alg:eag} using the notation $u^t = (\gamma_x^{-1/2}x^{k,t}, \gamma_y^{-1/2}y^{k,t})$ for $t \in \{-1,0,1,2,\ldots\}$. 

In addition, we replace the for-loop of \Cref{alg:eag} with the while-loop that iterates until the following condition is satisfied (see \cref{foam:line:while} of \Cref{alg:foam}):
\begin{equation}\label{foam:while}
	\begin{split}
			\gamma_x\sqn{a_x^k(x^{k,t}, y^{k,t}) + b_x^{k,t}} + \gamma_y\sqn{a_y^k(x^{k,t}, y^{k,t}) + b_y^{k,t}}\leq\qquad\qquad
			\\\leq \gamma_x^{-1}\sqn{x^{k,t} - x^{k,-1}} + \gamma_y^{-1}\sqn{y^{k,t} - y^{k,-1}}.
	\end{split}
\end{equation}
We also set the initial iterates to $x^{k,-1} = -\mu_x^{-1} z_g^k$ and $ y^{k,-1} = y_g^k$ on \cref{foam:line:xyi} of \Cref{alg:foam}, and 
use the output of the inner while-loop to compute vectors $(x_f^{k+1},y_f^{k+1},z_f^{k+1},w_f^{k+1})$
on \cref{foam:line:xyf,foam:line:zwf} of \Cref{alg:foam}.
Now, if we define parameters $\gamma_x,\gamma_y$ in the following way:
\begin{equation}\label{foam:gamma}
	\gamma_x = 8\mu_x^{-1},\qquad \gamma_y = \theta_y,
\end{equation}
then condition~\eqref{foam:while} on \cref{foam:line:while} of \Cref{alg:foam} becomes equivalent to condition~\eqref{pp:monteiro-svaiter} on \cref{pp:line:aux} of \Cref{alg:pp}. 
Hence, vectors $(x_f^{k+1},y_f^{k+1},z_f^{k+1},w_f^{k+1})$ computed on \cref{foam:line:xyf,foam:line:zwf} of \Cref{alg:foam} satisfy condition~\eqref{pp:monteiro-svaiter}, which implies that \Cref{alg:foam} is a special case of \Cref{alg:pp}.

\subsection{Complexity of the Algorithm}

It remains to establish the gradient evaluation complexity of \Cref{alg:foam}.
First, we need to estimate the number of iterations performed by the inner while-loop of \Cref{alg:foam}, which is equal to $t^k$ defined on \cref{foam:line:t} of \Cref{alg:foam}. Recall that the inner while-loop was constructed out of the lines of \Cref{alg:eag}. Hence, we can use \Cref{eag:thm} to provide an upper bound on $t^k$. This is done by the following lemma.
\begin{lemma}\label{lem:T}
	Assume the following choice of the parameters of \Cref{alg:foam}: stepsize $\lambda$ is defined by \eqref{eag:lambda}, parameter $M$ is defined by~\eqref{foam:M}, sequence $\{\beta_t\}_{t=0}^{\infty}$ is defined by~\eqref{eag:beta}, parameters $\gamma_x$ and $\gamma_y$ are defined by~\eqref{foam:gamma}. Then,  $t^k \leq T$, where $T$ is given as
	\begin{equation}\label{foam:T}
		T= \lceil 48\sqrt{2} \max\{8 L/\mu_x, 1 + \theta_y L\} \rceil -1.
	\end{equation}
\end{lemma}

Now, we are ready to provide the final gradient complexity of \Cref{alg:foam}. It is done by the following theorem.

\begin{theorem}\label{foam:thm}
	
	Let parameters of \Cref{alg:foam} be defined as follows: $\alpha =\min\left\{1, \sqrt{\theta_y\mu_y}\right\}$, $\theta_y = 8\mu_x^{-1}$, 	$\lambda = \left(2\sqrt{5}(1 + 8L/\mu_x)\right)^{-1}$, stepsizes $\eta_z$ and $\eta_y$ are defined by \eqref{pp:eta}, parameters  $\gamma_x$ and $\gamma_y$ are defined by \eqref{foam:gamma}, parameters $\{\beta_t\}_{t=0}^\infty$ are defined by \eqref{eag:beta}. Then, to find an $\epsilon$-accurate solution of problem~\eqref{eq:main}, \Cref{alg:foam} requires the following number of gradient evaluations:
	\begin{equation}
		\cO\left(\max\left\{\frac{L}{\mu_x}, \frac{L}{\sqrt{\mu_x\mu_y}}\right\}\log \frac{1}{\epsilon}\right).
	\end{equation}
	
\end{theorem}

\begin{corollary}
	Without loss of generality we can assume $\mu_x \geq \mu_y$, otherwise we just swap variables $x$ and $y$ in problem~\eqref{eq:main}. Hence, \Cref{alg:foam} has the following gradient evaluation complexity:
	\begin{equation}
		\cO\left(\frac{L}{\sqrt{\mu_x\mu_y}}\log \frac{1}{\epsilon}\right).
	\end{equation}
\end{corollary}

%
%

\newpage

\bibliographystyle{apalike}
\bibliography{reference.bib}

\begin{thebibliography}{}

\bibitem[Alkousa et~al., 2019]{alkousa2019accelerated}
Alkousa, M., Dvinskikh, D., Stonyakin, F., Gasnikov, A., and Kovalev, D.
  (2019).
\newblock Accelerated methods for composite non-bilinear saddle point problem.
\newblock {\em arXiv preprint arXiv:1906.03620}.

\bibitem[Arjovsky et~al., 2017]{arjovsky2017wasserstein}
Arjovsky, M., Chintala, S., and Bottou, L. (2017).
\newblock Wasserstein generative adversarial networks.
\newblock In {\em International conference on machine learning}, pages
  214--223. PMLR.

\bibitem[Ba{\c{s}}ar and Olsder, 1998]{bacsar1998dynamic}
Ba{\c{s}}ar, T. and Olsder, G.~J. (1998).
\newblock {\em Dynamic noncooperative game theory}.
\newblock SIAM.

\bibitem[Berger, 2013]{berger2013statistical}
Berger, J.~O. (2013).
\newblock {\em Statistical decision theory and Bayesian analysis}.
\newblock Springer Science \& Business Media.

\bibitem[Dai et~al., 2018]{dai2018sbeed}
Dai, B., Shaw, A., Li, L., Xiao, L., He, N., Liu, Z., Chen, J., and Song, L.
  (2018).
\newblock Sbeed: Convergent reinforcement learning with nonlinear function
  approximation.
\newblock In {\em International Conference on Machine Learning}, pages
  1125--1134. PMLR.

\bibitem[Du et~al., 2017]{du2017stochastic}
Du, S.~S., Chen, J., Li, L., Xiao, L., and Zhou, D. (2017).
\newblock Stochastic variance reduction methods for policy evaluation.
\newblock In {\em International Conference on Machine Learning}, pages
  1049--1058. PMLR.

\bibitem[Duchi et~al., 2011]{duchi2011adaptive}
Duchi, J., Hazan, E., and Singer, Y. (2011).
\newblock Adaptive subgradient methods for online learning and stochastic
  optimization.
\newblock {\em Journal of machine learning research}, 12(7).

\bibitem[Facchinei and Pang, 2003]{facchinei2003finite}
Facchinei, F. and Pang, J.-S. (2003).
\newblock {\em Finite-dimensional variational inequalities and complementarity
  problems}.
\newblock Springer.

\bibitem[Gidel et~al., 2018]{gidel2018variational}
Gidel, G., Berard, H., Vignoud, G., Vincent, P., and Lacoste-Julien, S. (2018).
\newblock A variational inequality perspective on generative adversarial
  networks.
\newblock {\em arXiv preprint arXiv:1802.10551}.

\bibitem[Goodfellow et~al., 2014]{goodfellow2014generative}
Goodfellow, I., Pouget-Abadie, J., Mirza, M., Xu, B., Warde-Farley, D., Ozair,
  S., Courville, A., and Bengio, Y. (2014).
\newblock Generative adversarial nets.
\newblock {\em Advances in neural information processing systems}, 27.

\bibitem[Ibrahim et~al., 2020]{ibrahim2020linear}
Ibrahim, A., Azizian, W., Gidel, G., and Mitliagkas, I. (2020).
\newblock Linear lower bounds and conditioning of differentiable games.
\newblock In {\em International conference on machine learning}, pages
  4583--4593. PMLR.

\bibitem[Kingma and Ba, 2014]{kingma2014adam}
Kingma, D.~P. and Ba, J. (2014).
\newblock Adam: A method for stochastic optimization.
\newblock {\em arXiv preprint arXiv:1412.6980}.

\bibitem[Kovalev et~al., 2021]{kovalev2021accelerated}
Kovalev, D., Gasnikov, A., and Richt{\'a}rik, P. (2021).
\newblock Accelerated primal-dual gradient method for smooth and convex-concave
  saddle-point problems with bilinear coupling.
\newblock {\em arXiv preprint arXiv:2112.15199}.

\bibitem[Lemar{\'e}chal and Sagastiz{\'a}bal, 1997]{lemarechal1997practical}
Lemar{\'e}chal, C. and Sagastiz{\'a}bal, C. (1997).
\newblock Practical aspects of the moreau--yosida regularization: Theoretical
  preliminaries.
\newblock {\em SIAM journal on optimization}, 7(2):367--385.

\bibitem[Lin et~al., 2020]{lin2020near}
Lin, T., Jin, C., and Jordan, M.~I. (2020).
\newblock Near-optimal algorithms for minimax optimization.
\newblock In {\em Conference on Learning Theory}, pages 2738--2779. PMLR.

\bibitem[Madry et~al., 2017]{madry2017towards}
Madry, A., Makelov, A., Schmidt, L., Tsipras, D., and Vladu, A. (2017).
\newblock Towards deep learning models resistant to adversarial attacks.
\newblock {\em arXiv preprint arXiv:1706.06083}.

\bibitem[Moreau, 1962]{moreau1962fonctions}
Moreau, J.~J. (1962).
\newblock Fonctions convexes duales et points proximaux dans un espace
  hilbertien.
\newblock {\em Comptes rendus hebdomadaires des s{\'e}ances de l'Acad{\'e}mie
  des sciences}, 255:2897--2899.

\bibitem[Nesterov, 2003]{nesterov2003introductory}
Nesterov, Y. (2003).
\newblock {\em Introductory lectures on convex optimization: A basic course},
  volume~87.
\newblock Springer Science \& Business Media.

\bibitem[Nesterov and Scrimali, 2006]{nesterov2006solving}
Nesterov, Y. and Scrimali, L. (2006).
\newblock Solving strongly monotone variational and quasi-variational
  inequalities.

\bibitem[Nesterov, 1983]{nesterov1983method}
Nesterov, Y.~E. (1983).
\newblock A method for solving the convex programming problem with convergence
  rate o (1/k\^{} 2).
\newblock In {\em Dokl. akad. nauk Sssr}, volume 269, pages 543--547.

\bibitem[Reddi et~al., 2019]{reddi2019convergence}
Reddi, S.~J., Kale, S., and Kumar, S. (2019).
\newblock On the convergence of adam and beyond.
\newblock {\em arXiv preprint arXiv:1904.09237}.

\bibitem[Rockafellar and Wets, 2009]{rockafellar2009variational}
Rockafellar, R.~T. and Wets, R. J.-B. (2009).
\newblock {\em Variational analysis}, volume 317.
\newblock Springer Science \& Business Media.

\bibitem[Roughgarden, 2010]{roughgarden2010algorithmic}
Roughgarden, T. (2010).
\newblock Algorithmic game theory.
\newblock {\em Communications of the ACM}, 53(7):78--86.

\bibitem[Sinha et~al., 2017]{sinha2017certifiable}
Sinha, A., Namkoong, H., and Duchi, J. (2017).
\newblock Certifiable distributional robustness with principled adversarial
  training.
\newblock {\em arXiv preprint arXiv:1710.10571}, 2.

\bibitem[Taskar et~al., 2005]{taskar2005structured}
Taskar, B., Lacoste-Julien, S., and Jordan, M. (2005).
\newblock Structured prediction via the extragradient method.
\newblock {\em Advances in neural information processing systems}, 18.

\bibitem[Tseng, 2000]{tseng2000modified}
Tseng, P. (2000).
\newblock A modified forward-backward splitting method for maximal monotone
  mappings.
\newblock {\em SIAM Journal on Control and Optimization}, 38(2):431--446.

\bibitem[Von~Neumann and Morgenstern, 1947]{von1947theory}
Von~Neumann, J. and Morgenstern, O. (1947).
\newblock Theory of games and economic behavior, 2nd rev.

\bibitem[Wang and Li, 2020]{wang2020improved}
Wang, Y. and Li, J. (2020).
\newblock Improved algorithms for convex-concave minimax optimization.
\newblock {\em Advances in Neural Information Processing Systems},
  33:4800--4810.

\bibitem[Xu et~al., 2009]{xu2009robustness}
Xu, H., Caramanis, C., and Mannor, S. (2009).
\newblock Robustness and regularization of support vector machines.
\newblock {\em Journal of machine learning research}, 10(7).

\bibitem[Yoon and Ryu, 2021]{yoon2021accelerated}
Yoon, T. and Ryu, E.~K. (2021).
\newblock Accelerated algorithms for smooth convex-concave minimax problems
  with o (1/k\^{} 2) rate on squared gradient norm.
\newblock In {\em International Conference on Machine Learning}, pages
  12098--12109. PMLR.

\bibitem[Yosida, 2012]{yosida2012functional}
Yosida, K. (2012).
\newblock {\em Functional analysis}.
\newblock Springer Science \& Business Media.

\bibitem[Zhang et~al., 2021]{zhang2021on}
Zhang, J., Hong, M., and Zhang, S. (2021).
\newblock On lower iteration complexity bounds for the convex concave saddle
  point problems.
\newblock {\em Mathematical Programming}.

\end{thebibliography}

\newpage

\appendix

\part*{Appendix}

\section{Proof of \Cref{lem:subdifferential}}

From the definition of $G(z,y)$ it follows that
\begin{equation}\label{eq:2}
	G(z,y) \geq \<z,x> - r(x) - \F(x,y) + g(y).
\end{equation}
Using $z - \nabla_x \F(x,y) \in \partial r(x)$, for arbitrary $\bar x \in \sX$ we get 
\begin{align*}
	r(\bar x) &\geq r(x) + \<z - \nabla_x \F(x,y), \bar x - x>
	\\&\geq
	r(x) + \<z,\bar x - x> + \F(x,y) - \F(\bar x, y),
\end{align*}
where we used the convexity of $\F(x,y)$ in $x$ in the last inequality.
After rearranging we get
\begin{equation*}
	\<z,x> - r(x) - \F(x,y) \geq \<z,\bar x> - r(\bar x) - \F(\bar x, y).
\end{equation*}
Now, we add $g(y)$ to both sides of the inequality and take supremum over $\bar x \in \sX$. This gives us
\begin{equation*}
	\<z,x> - r(x) - \F(x,y) + g(y) \geq \sup_{\bar x \in \sX}\left[\<z,\bar x> - r(\bar x) - \F(\bar x, y) + g(y)\right] = G(z,y),
\end{equation*}
which together with~\eqref{eq:2} implies $G(z,y) = \<z,x> - r(x) - \F(x,y) + g(y)$.

Next, we use $w + \nabla_y \F(x,y) \in \partial g(y)$, which for arbitrary $\bar y \in \sY$ implies
\begin{align*}
	g(\bar y) &\geq g(y) + \<w + \nabla_y \F(x,y), \bar y - y>
	\\&\geq g(y) + \<w, \bar y - y> + \F(x,\bar y) - \F(x,y),
\end{align*}
where we used the concavity of $\F(x,y)$ in $y$ in the last inequality. After rearranging we get
\begin{equation*}
	g(\bar y) - \F(x,\bar y)\geq g(y) - \F(x,y) +\<w, \bar y - y>.
\end{equation*}
Now, we choose arbitrary $\bar z \in \sX$ and add $\<x, \bar z> - r(x)$ to both sides of the inequality, which implies
\begin{align*}		
	g(\bar y) - \F(x,\bar y) + \<x, \bar z> - r(x)& \geq g(y) - \F(x,y) +\<w, \bar y - y> + \<x, \bar z> - r(x) \\&= 
	G(z,y) + \<x, \bar z - z> + \<w, \bar y - y>,
\end{align*}
where we used $G(z,y) = \<z,x> - r(x) - \F(x,y) + g(y)$. Further, using \eqref{eq:2} we get
\begin{equation*}
	G(z,y) + \<x, \bar z - z> + \<w, \bar y - y> \leq G(\bar z, \bar y),
\end{equation*}
which holds for arbitrary $\bar z \in \sX, \bar y \in \sY$. Hence, $(x,w) \in \partial G(z,y)$ by the definition of the subdifferential of a convex function. \qed

\section{Proof of \Cref{lem:solution}}

From the optimality conditions \eqref{eq:opt} and the definition of $\F(x,y)$, it follows that
\begin{align*}
	- \mu_x x^* -\nabla_x \F(x^*,y^*)  & \in \partial r(x^*),\\
	- \mu_y y^* + \nabla_y \F(x^*,y^*)  & \in \partial g(y^*).
\end{align*}
Using~\eqref{eq:zstar} we get
\begin{align*}
	z^* -\nabla_x \F(x^*,y^*)  & \in \partial r(x^*),\\
	- \mu_y y^* + \nabla_y \F(x^*,y^*)  & \in \partial g(y^*).
\end{align*}
Using \Cref{lem:subdifferential} we get 
\begin{equation*}
	(x^*, - \mu_y y^*) \in \partial G(z^*,y^*),
\end{equation*}
which together with \eqref{eq:zstar} implies
\begin{equation*}
	(-\mu_x^{-1} z^*, - \mu_y y^*) \in \partial G(z^*,y^*).
\end{equation*}
The latter condition implies $0 \in \partial P(z^*,y^*)$. Hence, $(z^*,y^*)$ is indeed a solution of problem~\eqref{eq:min}. The uniqueness of this solution is implied by the strong convexity of the function $P(z,y)$.
\qed

\section{Proof of \Cref{pp:thm}}

We start with proving two technical lemmas.

\begin{lemma}\label{pp:lem:1}
	Under conditions of \Cref{pp:thm} the following inequality holds:
	\begin{equation}
		\begin{split}
			\frac{1}{\eta_z}\sqn{z^{k+1} - z^*}
			&\leq
			\left(\frac{1}{\eta_z} - \mu_x^{-1}\right)\sqn{z^k - z^*}
			+\mu_x^{-1}\sqn{z_f^{k+1} - z^*}
			\\&
			+\frac{2}{\alpha}\<x_f^{k+1} + \mu_x^{-1} z_f^{k+1}, (1-\alpha)z_f^k + \alpha z^*- z_f^{k+1}>
			\\&
			+\frac{1}{\alpha}
			\left(
			\frac{8}{\mu_x}\sqn{z_f^{k+1} + \frac{\mu_x}{2}(x_f^{k+1} - \mu_x^{-1} z_g^k)}
			-\frac{\mu_x}{8}\sqn{x_f^{k+1} +\mu_x^{-1} z_g^k}
			\right).
		\end{split}
	\end{equation}
\end{lemma}
\begin{proof}
	Using \cref{pp:line:z} of \Cref{alg:pp} we get
	\begin{align*}
		\frac{1}{\eta_z}\sqn{z^{k+1} - z^*}
		&=
		\frac{1}{\eta_z}\sqn{z^k - z^*} + \frac{2}{\eta_z}\<z^{k+1} - z^k, z^k - z^*> + \frac{1}{\eta_z}\sqn{z^{k+1} - z^k}
		\\&=
		\frac{1}{\eta_z}\sqn{z^k - z^*}
		+\eta_z\sqn{\mu_x^{-1}(z_f^{k+1} - z^k) - (x_f^{k+1} + \mu_x^{-1}z_f^{k+1})}
		\\&
		+2\mu_x^{-1}\<z_f^{k+1} - z^k, z^k - z^*>
		-2\<x_f^{k+1} + \mu_x^{-1} z_f^{k+1}, z^k - z^*>.
	\end{align*}
	Using the parallelogram rule we get
	\begin{align*}
		\frac{1}{\eta_z}\sqn{z^{k+1} - z^*}
		&=
		\frac{1}{\eta_z}\sqn{z^k - z^*}
		+\eta_z\sqn{\mu_x^{-1}(z_f^{k+1} - z^k) - (x_f^{k+1} + \mu_x^{-1}z_f^{k+1})}
		\\&
		+\mu_x^{-1}\sqn{z_f^{k+1} - z^*}
		-\mu_x^{-1}\sqn{z^k - z^*}
		-\mu_x^{-1}\sqn{z_f^{k+1} - z^k}
		\\&
		-2\<x_f^{k+1} + \mu_x^{-1} z_f^{k+1}, z^k - z^*>.
	\end{align*}
	Using the inequality $\sqn{a+b} \leq 2\sqn{a} + 2\sqn{b}$ we get
	\begin{align*}
		\frac{1}{\eta_z}\sqn{z^{k+1} - z^*}
		&\leq
		\frac{1}{\eta_z}\sqn{z^k - z^*}
		+2\eta_z\sqn{x_f^{k+1} + \mu_x^{-1}z_f^{k+1}}
		+2\eta_z\mu_x^{-2}\sqn{z_f^{k+1} - z^k}
		\\&
		+\mu_x^{-1}\sqn{z_f^{k+1} - z^*}
		-\mu_x^{-1}\sqn{z^k - z^*}
		-\mu_x^{-1}\sqn{z_f^{k+1} - z^k}
		\\&
		-2\<x_f^{k+1} + \mu_x^{-1} z_f^{k+1}, z^k - z^*>.
	\end{align*}
	Using the definition of $\eta_z$ we get
	\begin{align*}
		\frac{1}{\eta_z}\sqn{z^{k+1} - z^*}
		&\leq
		\frac{1}{\eta_z}\sqn{z^k - z^*}
		+\mu_x\sqn{x_f^{k+1} + \mu_x^{-1}z_f^{k+1}}
		+\mu_x^{-1}\sqn{z_f^{k+1} - z^k}
		\\&
		+\mu_x^{-1}\sqn{z_f^{k+1} - z^*}
		-\mu_x^{-1}\sqn{z^k - z^*}
		-\mu_x^{-1}\sqn{z_f^{k+1} - z^k}
		\\&
		-2\<x_f^{k+1} + \mu_x^{-1} z_f^{k+1}, z^k - z^*>
		\\&=
		\left(\frac{1}{\eta_z} - \mu_x^{-1}\right)\sqn{z^k - z^*}
		+\mu_x\sqn{x_f^{k+1} + \mu_x^{-1}z_f^{k+1}}
		\\&
		+\mu_x^{-1}\sqn{z_f^{k+1} - z^*}
		-2\<x_f^{k+1} + \mu_x^{-1} z_f^{k+1}, z^k - z^*>.
	\end{align*}
	From \cref{pp:line:convex} of \Cref{alg:pp} we get $z^k = \alpha^{-1}z_g^k - (1-\alpha)\alpha^{-1}z_f^k$ which implies
	\begin{align*}
		\frac{1}{\eta_z}\sqn{z^{k+1} - z^*}
		&\leq
		\left(\frac{1}{\eta_z} - \mu_x^{-1}\right)\sqn{z^k - z^*}
		+\mu_x\sqn{x_f^{k+1} + \mu_x^{-1}z_f^{k+1}}
		\\&
		+\mu_x^{-1}\sqn{z_f^{k+1} - z^*}
		-2\<x_f^{k+1} + \mu_x^{-1} z_f^{k+1}, \alpha^{-1}z_g^k - (1-\alpha)\alpha^{-1}z_f^k - z^*>
		\\&=
		\left(\frac{1}{\eta_z} - \mu_x^{-1}\right)\sqn{z^k - z^*}
		+\mu_x\sqn{x_f^{k+1} + \mu_x^{-1}z_f^{k+1}}
		+\mu_x^{-1}\sqn{z_f^{k+1} - z^*}
		\\&
		+\frac{2}{\alpha}\<x_f^{k+1} + \mu_x^{-1} z_f^{k+1}, (1-\alpha)z_f^k + \alpha z^*- z_f^{k+1}>
		\\&
		+\frac{2}{\alpha}\<x_f^{k+1} + \mu_x^{-1} z_f^{k+1}, z_f^{k+1} - z_g^k >
		\\&=
		\left(\frac{1}{\eta_z} - \mu_x^{-1}\right)\sqn{z^k - z^*}
		+\mu_x\sqn{x_f^{k+1} + \mu_x^{-1}z_f^{k+1}}
		+\mu_x^{-1}\sqn{z_f^{k+1} - z^*}
		\\&
		+\frac{2}{\alpha}\<x_f^{k+1} + \mu_x^{-1} z_f^{k+1}, (1-\alpha)z_f^k + \alpha z^*- z_f^{k+1}>
		\\&
		+\frac{2 \mu_x }{\alpha}\< x_f^{k+1} +\mu_x^{-1}z_f^{k+1}, \mu_x^{-1}(z_f^{k+1} - z_g^k) >.
	\end{align*}
	Using the parallelogram rule we get
	\begin{align*}
		\frac{1}{\eta_z}\sqn{z^{k+1} - z^*}
		&\leq
		\left(\frac{1}{\eta_z} - \mu_x^{-1}\right)\sqn{z^k - z^*}
		+\mu_x\sqn{x_f^{k+1} + \mu_x^{-1}z_f^{k+1}}
		+\mu_x^{-1}\sqn{z_f^{k+1} - z^*}
		\\&
		+\frac{2}{\alpha}\<x_f^{k+1} + \mu_x^{-1} z_f^{k+1}, (1-\alpha)z_f^k + \alpha z^*- z_f^{k+1}>
		\\&
		+\frac{\mu_x}{\alpha}\sqn{x_f^{k+1} +2\mu_x^{-1}z_f^{k+1} - \mu_x^{-1}z_g^k}
		-\frac{\mu_x^{-1}}{\alpha}\sqn{z_f^{k+1} - z_g^k}
		-\frac{\mu_x }{\alpha}\sqn{x_f^{k+1} + \mu_x^{-1} z_g^k}
		\\&=		
		\left(\frac{1}{\eta_z} - \mu_x^{-1}\right)\sqn{z^k - z^*}
		+\mu_x(1-\alpha^{-1})\sqn{x_f^{k+1} + \mu_x^{-1}z_f^{k+1}}
		+\mu_x^{-1}\sqn{z_f^{k+1} - z^*}
		\\&
		+\frac{2}{\alpha}\<x_f^{k+1} + \mu_x^{-1} z_f^{k+1}, (1-\alpha)z_f^k + \alpha z^*- z_f^{k+1}>
		\\&
		+\frac{4\mu_x^{-1}}{\alpha}\sqn{z_f^{k+1} + \frac{\mu_x}{2}(x_f^{k+1} - \mu_x^{-1} z_g^k)}
		-\frac{\mu_x^{-1}}{\alpha}\sqn{z_f^{k+1} - z_g^k}.
	\end{align*}
	Using the fact that $\alpha^{-1} \geq 1$ we get
	\begin{align*}
		\frac{1}{\eta_z}\sqn{z^{k+1} - z^*}
		&\leq
		\left(\frac{1}{\eta_z} - \mu_x^{-1}\right)\sqn{z^k - z^*}
		+\mu_x^{-1}\sqn{z_f^{k+1} - z^*}
		\\&
		+\frac{2}{\alpha}\<x_f^{k+1} + \mu_x^{-1} z_f^{k+1}, (1-\alpha)z_f^k + \alpha z^*- z_f^{k+1}>
		\\&
		+\frac{4\mu_x^{-1}}{\alpha}\sqn{z_f^{k+1} + \frac{\mu_x}{2}(x_f^{k+1} - \mu_x^{-1} z_g^k)}
		-\frac{\mu_x^{-1}}{\alpha}\sqn{z_f^{k+1} - z_g^k}
		\\&=
		\left(\frac{1}{\eta_z} - \mu_x^{-1}\right)\sqn{z^k - z^*}
		+\mu_x^{-1}\sqn{z_f^{k+1} - z^*}
		\\&
		+\frac{2}{\alpha}\<x_f^{k+1} + \mu_x^{-1} z_f^{k+1}, (1-\alpha)z_f^k + \alpha z^*- z_f^{k+1}>
		\\&
		+\frac{4\mu_x^{-1}}{\alpha}\sqn{z_f^{k+1} + \frac{\mu_x}{2}(x_f^{k+1} - \mu_x^{-1} z_g^k)}
		\\&
		-\frac{\mu_x^{-1}}{\alpha}\sqn{z_f^{k+1} + \frac{\mu_x}{2}(x_f^{k+1} - \mu_x^{-1} z_g^k) - \frac{\mu_x}{2}(x_f^{k+1} - \mu_x^{-1} z_g^k)- z_g^k}.
	\end{align*}
	Using the inequality $-\sqn{a+b} \leq  \sqn{b}-\frac{1}{2}\sqn{a} $ we get
	\begin{align*}
		\frac{1}{\eta_z}\sqn{z^{k+1} - z^*}
		&\leq
		\left(\frac{1}{\eta_z} - \mu_x^{-1}\right)\sqn{z^k - z^*}
		+\mu_x^{-1}\sqn{z_f^{k+1} - z^*}
		\\&
		+\frac{2}{\alpha}\<x_f^{k+1} + \mu_x^{-1} z_f^{k+1}, (1-\alpha)z_f^k + \alpha z^*- z_f^{k+1}>
		\\&
		+\frac{4\mu_x^{-1}}{\alpha}\sqn{z_f^{k+1} + \frac{\mu_x}{2}(x_f^{k+1} - \mu_x^{-1} z_g^k)}
		\\&
		+\frac{\mu_x^{-1}}{\alpha}\sqn{z_f^{k+1} + \frac{\mu_x}{2}(x_f^{k+1} - \mu_x^{-1} z_g^k)}
		-\frac{\mu_x^{-1}}{2\alpha}\sqn{\frac{\mu_x}{2}(x_f^{k+1} -\mu_x^{-1} z_g^k)+ z_g^k}
		\\&=
		\left(\frac{1}{\eta_z} - \mu_x^{-1}\right)\sqn{z^k - z^*}
		+\mu_x^{-1}\sqn{z_f^{k+1} - z^*}
		\\&
		+\frac{2}{\alpha}\<x_f^{k+1} + \mu_x^{-1} z_f^{k+1}, (1-\alpha)z_f^k + \alpha z^*- z_f^{k+1}>
		\\&
		+\frac{5\mu_x^{-1}}{\alpha}\sqn{z_f^{k+1} + \frac{\mu_x}{2}(x_f^{k+1} - \mu_x^{-1} z_g^k)}
		-\frac{\mu_x}{8\alpha}\sqn{x_f^{k+1} +\mu_x^{-1} z_g^k}.
	\end{align*}
	Using the inequality $5\leq 8$ concludes the proof
\end{proof}

\begin{lemma}\label{pp:lem:2}
	Under conditions of \Cref{pp:thm} the following inequality holds:
	\begin{equation}
		\begin{split}
			\frac{1}{\eta_y}\sqn{y^{k+1} - y^*}
			&\leq
			\left(\frac{1}{\eta_y} - \mu_y\right)\sqn{y^k - y^*}
			+\mu_y\sqn{y_f^{k+1} - y^*}
			\\&
			+\frac{2}{\alpha}\<w_f^{k+1} + \mu_y y_f^{k+1}, (1-\alpha)y_f^k + \alpha y^* - y_f^{k+1}>
			\\&
			+\frac{1}{\alpha}
			\left(
			\theta_y\sqn{w_f^{k+1} + \mu_y y_f^{k+1} + \theta_y^{-1} (y_f^{k+1} - y_g^k)}
			-\theta_y^{-1}\sqn{y_f^{k+1} - y_g^k}
			\right).
		\end{split}
	\end{equation}
\end{lemma}
\begin{proof}
	Using \cref{pp:line:y} of \Cref{alg:pp} we get
	\begin{align*}
		\frac{1}{\eta_y}\sqn{y^{k+1} - y^*}
		&=
		\frac{1}{\eta_y}\sqn{y^k - y^*}
		+\frac{2}{\eta_y}\<y^{k+1} - y^k, y^k - y^*>
		+\frac{1}{\eta_y}\sqn{y^{k+1} - y^*}
		\\&=
		\frac{1}{\eta_y}\sqn{y^k - y^*} + \eta_y\sqn{\mu_y(y_f^{k+1} - y^k) - (w_f^{k+1} + \mu_y y_f^{k+1})}
		\\&
		+2\mu_y\<y_f^{k+1} - y^k,y^k - y^*>
		-2\<w_f^{k+1} + \mu_y y_f^{k+1}, y^k - y^*>.
	\end{align*}
	Using the parallelogram rule we get
	\begin{align*}
		\frac{1}{\eta_y}\sqn{y^{k+1} - y^*}
		&=
		\frac{1}{\eta_y}\sqn{y^k - y^*} + \eta_y\sqn{\mu_y(y_f^{k+1} - y^k) - (w_f^{k+1} + \mu_y y_f^{k+1})}
		\\&
		\mu_y\sqn{y_f^{k+1} - y^*} - \mu_y\sqn{y^k - y^*} - \mu_y\sqn{y_f^{k+1} - y^k}
		\\&
		-2\<w_f^{k+1} + \mu_y y_f^{k+1}, y^k - y^*>.
	\end{align*}
	Using the inequality $\sqn{a+b} \leq 2\sqn{a} + 2\sqn{b}$ we get
	\begin{align*}
		\frac{1}{\eta_y}\sqn{y^{k+1} - y^*}
		&=
		\frac{1}{\eta_y}\sqn{y^k - y^*}
		+2\eta_y\sqn{w_f^{k+1} + \mu_y y_f^{k+1}}
		+2\eta_y\mu_y^2\sqn{y_f^{k+1} - y^k}
		\\&
		+\mu_y\sqn{y_f^{k+1} - y^*} - \mu_y\sqn{y^k - y^*} - \mu_y\sqn{y_f^{k+1} - y^k}
		\\&
		-2\<w_f^{k+1} + \mu_y y_f^{k+1}, y^k - y^*>.
	\end{align*}
	Using the definition of $\eta_y$ we get
	\begin{align*}
		\frac{1}{\eta_y}\sqn{y^{k+1} - y^*}
		&\leq
		\frac{1}{\eta_y}\sqn{y^k - y^*}
		+2\eta_y\sqn{w_f^{k+1} + \mu_y y_f^{k+1}}
		+\mu_y\sqn{y_f^{k+1} - y^k}
		\\&
		+\mu_y\sqn{y_f^{k+1} - y^*} - \mu_y\sqn{y^k - y^*} - \mu_y\sqn{y_f^{k+1} - y^k}
		\\&
		-2\<w_f^{k+1} + \mu_y y_f^{k+1}, y^k - y^*>
		\\&=
		\left(\frac{1}{\eta_y} - \mu_y\right)\sqn{y^k - y^*}
		+2\eta_y\sqn{w_f^{k+1} + \mu_y y_f^{k+1}}
		\\&
		+\mu_y\sqn{y_f^{k+1} - y^*}
		-2\<w_f^{k+1} + \mu_y y_f^{k+1}, y^k - y^*>.
	\end{align*}
	From \cref{pp:line:convex} of \Cref{alg:pp} we get $y^k = \alpha^{-1}y_g^k - (1-\alpha)\alpha^{-1}y_f^k$, which implies
	\begin{align*}
		\frac{1}{\eta_y}\sqn{y^{k+1} - y^*}
		&\leq
		\left(\frac{1}{\eta_y} - \mu_y\right)\sqn{y^k - y^*}
		+2\eta_y\sqn{w_f^{k+1} + \mu_y y_f^{k+1}}
		\\&
		+\mu_y\sqn{y_f^{k+1} - y^*}
		-2\<w_f^{k+1} + \mu_y y_f^{k+1}, \alpha^{-1}y_g^k - (1-\alpha)\alpha^{-1}y_f^k - y^*>
		\\&=
		\left(\frac{1}{\eta_y} - \mu_y\right)\sqn{y^k - y^*}
		+2\eta_y\sqn{w_f^{k+1} + \mu_y y_f^{k+1}}
		+\mu_y\sqn{y_f^{k+1} - y^*}
		\\&
		+\frac{2}{\alpha}\<w_f^{k+1} + \mu_y y_f^{k+1}, (1-\alpha)y_f^k + \alpha y^* - y_f^{k+1}>
		\\&
		+\frac{2}{\alpha}\<w_f^{k+1} + \mu_y y_f^{k+1}, y_f^{k+1} - y_g^k>
		\\&=
		\left(\frac{1}{\eta_y} - \mu_y\right)\sqn{y^k - y^*}
		+2\eta_y\sqn{w_f^{k+1} + \mu_y y_f^{k+1}}
		+\mu_y\sqn{y_f^{k+1} - y^*}
		\\&
		+\frac{2}{\alpha}\<w_f^{k+1} + \mu_y y_f^{k+1}, (1-\alpha)y_f^k + \alpha y^* - y_f^{k+1}>
		\\&
		+\frac{2\theta_y}{\alpha}\<w_f^{k+1} + \mu_y y_f^{k+1}, \theta_y^{-1} (y_f^{k+1} - y_g^k)>
	\end{align*}
	Using the parallelogram rule we get
	\begin{align*}
		\frac{1}{\eta_y}\sqn{y^{k+1} - y^*}
		&\leq
		\left(\frac{1}{\eta_y} - \mu_y\right)\sqn{y^k - y^*}
		+2\eta_y\sqn{w_f^{k+1} + \mu_y y_f^{k+1}}
		+\mu_y\sqn{y_f^{k+1} - y^*}
		\\&
		+\frac{2}{\alpha}\<w_f^{k+1} + \mu_y y_f^{k+1}, (1-\alpha)y_f^k + \alpha y^* - y_f^{k+1}>
		\\&
		+\frac{\theta_y}{\alpha}\sqn{w_f^{k+1} + \mu_y y_f^{k+1} + \theta_y^{-1} (y_f^{k+1} - y_g^k)}
		\\&
		-\frac{\theta_y}{\alpha}\sqn{w_f^{k+1} + \mu_y y_f^{k+1}}
		-\frac{\theta_y^{-1}}{\alpha}\sqn{y_f^{k+1} - y_g^k}
		\\&=
		\left(\frac{1}{\eta_y} - \mu_y\right)\sqn{y^k - y^*}
		+(2\eta_y - \alpha^{-1}\theta_y)\sqn{w_f^{k+1} + \mu_y y_f^{k+1}}
		+\mu_y\sqn{y_f^{k+1} - y^*}
		\\&
		+\frac{2}{\alpha}\<w_f^{k+1} + \mu_y y_f^{k+1}, (1-\alpha)y_f^k + \alpha y^* - y_f^{k+1}>
		\\&
		+\frac{\theta_y}{\alpha}\sqn{w_f^{k+1} + \mu_y y_f^{k+1} + \theta_y^{-1} (y_f^{k+1} - y_g^k)}
		-\frac{\theta_y^{-1}}{\alpha}\sqn{y_f^{k+1} - y_g^k}.
	\end{align*}
	From the definition of $\eta_y$ it follows that $2\eta_y \leq \alpha^{-1}\theta_y$. Hence,
	\begin{align*}
		\frac{1}{\eta_y}\sqn{y^{k+1} - y^*}
		&\leq
		\left(\frac{1}{\eta_y} - \mu_y\right)\sqn{y^k - y^*}
		+\mu_y\sqn{y_f^{k+1} - y^*}
		\\&
		+\frac{2}{\alpha}\<w_f^{k+1} + \mu_y y_f^{k+1}, (1-\alpha)y_f^k + \alpha y^* - y_f^{k+1}>
		\\&
		+\frac{\theta_y}{\alpha}\sqn{w_f^{k+1} + \mu_y y_f^{k+1} + \theta_y^{-1} (y_f^{k+1} - y_g^k)}
		-\frac{\theta_y^{-1}}{\alpha}\sqn{y_f^{k+1} - y_g^k}.
	\end{align*}
\end{proof}

\begin{lemma}\label{pp:lem:3}
	Under conditions of \Cref{pp:thm}, let $\cL^k$ be the following Lyapunov function
	\begin{equation}
		\cL^k = \frac{1}{\eta_z}\sqn{z^k - z^*} + \frac{1}{\eta_y}\sqn{y^k - y^*} + \frac{2}{\alpha}\left(P(z_f^k,y_f^k) - P(z^*,y^*)\right).
	\end{equation}
	Then, the following inequality holds
	\begin{equation}
		\cL^{k+1} \leq \left(1 - \max\left\{\frac{2}{\alpha}, \frac{2\alpha}{\theta_y\mu_y}\right\}^{-1}\right)\cL^k.
	\end{equation}
\end{lemma}
\begin{proof}
	We start with combining \Cref{pp:lem:1,pp:lem:2} and get
	\begin{align*}
		\text{(DISTANCE)}
		&\leq
		\left(\frac{1}{\eta_z} - \mu_x^{-1}\right)\sqn{z^k - z^*}
		+\mu_x^{-1}\sqn{z_f^{k+1} - z^*}
		\\&
		+\frac{2}{\alpha}\<x_f^{k+1} + \mu_x^{-1} z_f^{k+1}, (1-\alpha)z_f^k + \alpha z^*- z_f^{k+1}>
		\\&
		+\frac{1}{\alpha}
		\left(
		\frac{8}{\mu_x}\sqn{z_f^{k+1} + \frac{\mu_x}{2}(x_f^{k+1} - \mu_x^{-1} z_g^k)}
		-\frac{\mu_x}{8}\sqn{x_f^{k+1} +\mu_x^{-1} z_g^k}
		\right)
		\\&
		+\left(\frac{1}{\eta_y} - \mu_y\right)\sqn{y^k - y^*}
		+\mu_y\sqn{y_f^{k+1} - y^*}
		\\&
		+\frac{2}{\alpha}\<w_f^{k+1} + \mu_y y_f^{k+1}, (1-\alpha)y_f^k + \alpha y^* - y_f^{k+1}>
		\\&
		+\frac{1}{\alpha}
		\left(
		\theta_y\sqn{w_f^{k+1} + \mu_y y_f^{k+1} + \theta_y^{-1} (y_f^{k+1} - y_g^k)}
		-\theta_y^{-1}\sqn{y_f^{k+1} - y_g^k}
		\right).
	\end{align*}
	where $\text{(DISTANCE)}$ is defined as
	\begin{equation*}
		\text{(DISTANCE)} = \frac{1}{\eta_z}\sqn{z^{k+1} - z^*} + \frac{1}{\eta_y}\sqn{y^{k+1} - y^*}.
	\end{equation*}
	Using condition \eqref{pp:monteiro-svaiter} we get
	\begin{align*}
		\text{(DISTANCE)}
		&\leq
		\left(\frac{1}{\eta_z} - \mu_x^{-1}\right)\sqn{z^k - z^*}
		+\left(\frac{1}{\eta_y} - \mu_y\right)\sqn{y^k - y^*}
		\\&
		+\mu_x^{-1}\sqn{z_f^{k+1} - z^*}
		+\mu_y\sqn{y_f^{k+1} - y^*}
		\\&
		+\frac{2}{\alpha}\<x_f^{k+1} + \mu_x^{-1} z_f^{k+1}, (1-\alpha)z_f^k + \alpha z^*- z_f^{k+1}>
		\\&
		+\frac{2}{\alpha}\<w_f^{k+1} + \mu_y y_f^{k+1}, (1-\alpha)y_f^k + \alpha y^* - y_f^{k+1}>
		\\&=
		\left(\frac{1}{\eta_z} - \mu_x^{-1}\right)\sqn{z^k - z^*}
		+\left(\frac{1}{\eta_y} - \mu_y\right)\sqn{y^k - y^*}
		\\&
		+\mu_x^{-1}\sqn{z_f^{k+1} - z^*}
		+\mu_y\sqn{y_f^{k+1} - y^*}
		\\&
		+\frac{2(1-\alpha)}{\alpha}\Dotprod{\vect{x_f^{k+1}\\w_f^{k+1}} + \vect{\mu_x^{-1} z_f^{k+1}\\\mu_y y_f^{k+1}},\vect{z_f^k \\y_f^k } - \vect{z_f^{k+1}\\y_f^{k+1}}}
		\\&
		+2\Dotprod{\vect{x_f^{k+1}\\w_f^{k+1}} + \vect{\mu_x^{-1} z_f^{k+1}\\\mu_y y_f^{k+1}},\vect{z^*\\y^* } - \vect{z_f^{k+1}\\y_f^{k+1}}}.
	\end{align*}
	Using condition~\eqref{pp:monteiro-svaiter} and \Cref{lem:subdifferential} we get
	\begin{equation*}
		\vect{x_f^{k+1}\\w_f^{k+1}}  \in \partial G(z_f^{k+1}, y_f^{k+1}).
	\end{equation*}
	Using the definition of function $P(z,y)$ we get
	\begin{equation}
		\vect{x_f^{k+1}\\w_f^{k+1}} + \vect{\mu_x^{-1} z_f^{k+1}\\\mu_y y_f^{k+1}} \in \partial P(z_f^{k+1}, y_f^{k+1}).
	\end{equation}
	Hence, using the strong convexity of Function $P(z,y)$ we get
	\begin{align*}
		\text{(DISTANCE)}
		&\leq
		\left(\frac{1}{\eta_z} - \mu_x^{-1}\right)\sqn{z^k - z^*}
		+\left(\frac{1}{\eta_y} - \mu_y\right)\sqn{y^k - y^*}
		\\&
		+\mu_x^{-1}\sqn{z_f^{k+1} - z^*}
		+\mu_y\sqn{y_f^{k+1} - y^*}
		\\&
		+\frac{2(1-\alpha)}{\alpha}\left(
		P(z_f^k,y_f^k) - P(z_f^{k+1}, y_f^{k+1})
		\right)
		\\&
		+2\left(P(z^*,y^*) - P(z_f^{k+1},y_f^{k+1}) - \frac{\mu_x^{-1}}{2}\sqn{z_f^{k+1} - z^*} - \frac{\mu_y}{2}\sqn{y_f^{k+1} - y^*}\right)
		\\&=
		\left(\frac{1}{\eta_z} - \mu_x^{-1}\right)\sqn{z^k - z^*}
		+\left(\frac{1}{\eta_y} - \mu_y\right)\sqn{y^k - y^*}
		\\&
		+\frac{2(1-\alpha)}{\alpha}\left(
		P(z_f^k,y_f^k) - P(z^*,y^*)
		\right)
		-\frac{2}{\alpha}\left(
		P(z_f^{k+1},y_f^{k+1}) - P(z^*,y^*)
		\right).
	\end{align*}
	After rearranging and using the definition of $\text{(DISTANCE)}$ and the definition of $\cL^k$ we get
	\begin{align*}
		\cL^{k+1} &\leq
		\left(\frac{1}{\eta_z} - \mu_x^{-1}\right)\sqn{z^k - z^*}
		+\left(\frac{1}{\eta_y} - \mu_y\right)\sqn{y^k - y^*}
		\\&
		+\frac{2(1-\alpha)}{\alpha}\left(
		P(z_f^k,y_f^k) - P(z^*,y^*)
		\right)
		\\&\leq
		\left(1 - \max\left\{2, \frac{1}{\alpha}, \frac{2\alpha}{\theta_y\mu_y}\right\}^{-1}\right)\cL^k
		\\&\leq
		\left(1 - \max\left\{\frac{2}{\alpha}, \frac{2\alpha}{\theta_y\mu_y}\right\}^{-1}\right)\cL^k.
	\end{align*}
\end{proof}

Now, we are ready to prove \Cref{pp:thm}.

\begin{proof}[\bf Proof of \Cref{pp:thm}]
	After unrolling the recurrence from \Cref{pp:lem:3} we get
	\begin{equation*}
		\frac{1}{\eta_z}\sqn{z^K - z^*} + \frac{1}{\eta_y}\sqn{y^K - y^*} \leq 	\left(1 - \max\left\{\frac{2}{\alpha}, \frac{2\alpha}{\theta_y\mu_y}\right\}^{-1}\right)^K\cL^0.
	\end{equation*}
	Using \Cref{lem:solution} we get
	\begin{align*}
		\left(1 - \max\left\{\frac{2}{\alpha}, \frac{2\alpha}{\theta_y\mu_y}\right\}^{-1}\right)^K\cL^0 &\geq \frac{1}{\eta_z}\sqn{z^K + \mu_x z^*} + \frac{1}{\eta_y}\sqn{y^K - y^*} 
		\\&=
		\frac{\mu_x^2}{\eta_z}\sqn{\mu_x^{-1}z^K + z^*} + \frac{1}{\eta_y}\sqn{y^K - y^*} 
		\\&\geq
		\min \left\{ \eta_z^{-1} \mu_x^2, \eta_y^{-1}\right\} \left(\sqn{\mu_x^{-1}z^K + z^*}+\sqn{y^K - y^*} \right).
	\end{align*}
	After rearranging we get 
	\begin{equation*}
		\sqn{\mu_x^{-1}z^K + z^*}+\sqn{y^K - y^*}  \leq C\left(1 - \max\left\{\frac{2}{\alpha}, \frac{2\alpha}{\theta_y\mu_y}\right\}^{-1}\right)^K,
	\end{equation*}
	where $C$ is given as
	\begin{equation*}
		C = \max\{\eta_z \mu_x^{-2}, \eta_y\}\cL^0.
	\end{equation*}
	Hence, $(\sqn{\mu_x^{-1}z^K + z^*}+\sqn{y^K - y^*} )\leq \epsilon$ if the number of iterations $K$ satisfies 
	\begin{equation*}
		K \geq \max\left\{\frac{2}{\alpha}, \frac{2\alpha}{\theta_y\mu_y}\right\} \log \frac{C}{\epsilon}.
	\end{equation*}
	This concludes the proof.

\end{proof}

\section{Proof of \Cref{eag:thm}}

We start the proof with three technical lemmas.

\begin{lemma}\label{eag:lem:1}
	Under assumptions of \Cref{eag:thm} the following inequality holds:
	\begin{equation}
		96M^2\sqn{u^0 - u^*} + 12\sqn{a^0 + b^0}
		\leq
		288M^2\sqn{u^{-1} - u^* }.
	\end{equation}
\end{lemma}
\begin{proof}
	From \Cref{eag:line:0} of \Cref{alg:eag} it follows that
	\begin{equation*}
		u^0 = \J_{\lambda B} (u^{-1} -\lambda a^{-1}),
	\end{equation*}
	where $a^{-1} = A(u^{-1})$. Vector $u^*$ is the solution to problem~\eqref{eq:inclusion}. Hence, there exists $b^* \in B(u^*)$ such that $a^* + b^* = 0$, where $a^* = A(u^*)$. This implies
	\begin{equation*}
		u^* = \J_{\lambda B}(u^* - \lambda a^*).
	\end{equation*}
	Using the firm non-expansiveness of $\J_{\lambda B}$ we get
	\begin{align*}
		\sqn{u^0 - u^*} 
		&=
		\sqn{\J_{\lambda B}(u^{-1} - \lambda a^{-1}) - \J_{\lambda B}(u^* - \lambda a^*)}
		\\&\leq 
		\sqn{(u^{-1} - \lambda a^{-1}) - (u^* - \lambda a^*)}	
		\\&
		-\sqn{\J_{\lambda B}(u^{-1} - \lambda a^{-1}) - \J_{\lambda B}(u^* - \lambda a^*) - (u^{-1} - \lambda a^{-1}) + (u^* - \lambda a^*)}
		\\&=
		\sqn{u^{-1} - u^*  - \lambda (a^{-1} - a^*)}
		-\sqn{(u^* - \lambda a^* - u^*)- (u^{-1} - \lambda a^{-1} - u^0)}
		\\&=
		\sqn{u^{-1} - u^*  - \lambda (a^{-1} - a^*)}
		-\sqn{\lambda a^*+(u^{-1} - \lambda a^{-1} - u^0)}.
	\end{align*}
	Using \cref{eag:line:3} of \Cref{alg:eag} we get
	\begin{align*}
		\sqn{u^0 - u^*} 
		&\leq
		\sqn{u^{-1} - u^*  - \lambda (a^{-1} - a^*)}
		-\sqn{\lambda a^*+ \lambda b^0}
		\\&=
		\sqn{u^{-1} - u^*  - \lambda (a^{-1} - a^*)}
		-\sqn{\lambda (a^* - a^0)+ \lambda(a^0 + b^0)}.
	\end{align*}
	Using the inequality $\sqn{\lambda (a^* - a^0)+ \lambda(a^0 + b^0)} \geq \frac{\lambda^2}{2}\sqn{a^0 + b^0} - \lambda^2\sqn{a^0 - a^*}$ we get
	\begin{align*}
		\sqn{u^0 - u^*} 
		&\leq
		\sqn{u^{-1} - u^*  - \lambda (a^{-1} - a^*)}
		-\frac{\lambda^2}{2}\sqn{a^0 + b^0} + \lambda^2\sqn{a^0 - a^*}.
	\end{align*}
	Using the inequality $\sqn{u^{-1} - u^*  - \lambda (a^{-1} - a^*)} \leq 2\sqn{u^{-1} - u^* }+2\lambda^2\sqn{a^{-1} - a^*}$ we get
	\begin{align*}
		\sqn{u^0 - u^*} 
		&\leq
		2\sqn{u^{-1} - u^* }+2\lambda^2\sqn{a^{-1} - a^*}
		-\frac{\lambda^2}{2}\sqn{a^0 + b^0} + \lambda^2\sqn{a^0 - a^*}.
	\end{align*}
	Using the $M$-Lipschitzness of $A(u)$ we get
	\begin{align*}
		\sqn{u^0 - u^*} 
		&\leq
		2\sqn{u^{-1} - u^* }+2\lambda^2M^2\sqn{u^{-1}-u^*}
		-\frac{\lambda^2}{2}\sqn{a^0 + b^0} + \lambda^2M^2\sqn{u^0 - u^*}.
	\end{align*}
	After rearranging we get
	\begin{align*}
		(1 - \lambda^2 M^2)\sqn{u^0 - u^*} + \frac{\lambda^2}{2}\sqn{a^0 + b^0}
		&\leq
		2(1 + \lambda^2 M^2)\sqn{u^{-1} - u^* }.
	\end{align*}
	Plugging the definition of $\lambda$  gives
	\begin{align*}
		\frac{4}{5}\sqn{u^0 - u^*} + \frac{1}{10M^2}\sqn{a^0 + b^0}
		&\leq
		\frac{12}{5}\sqn{u^{-1} - u^* }.
	\end{align*}
	Multiplying both sides of the inequality by $120M^2$ gives
	\begin{align*}
		96M^2\sqn{u^0 - u^*} + 12\sqn{a^0 + b^0}
		&\leq
		M^2\sqn{u^{-1} - u^* }.
	\end{align*}
\end{proof}

\begin{lemma}\label{eag:lem:2}
	Under conditions of \Cref{eag:thm} the following equality holds:
	\begin{equation}
		\prod_{t=0}^{T-1} (1-\beta_t) = \frac{2}{(T+1)(T+2)}.
	\end{equation}
\end{lemma}

\begin{proof}
	\begin{align*}
		\prod_{t=0}^{T-1} (1-\beta_t)
		&=
		\prod_{t=0}^{T-1} \left(1 - \frac{2}{t+3}\right)
		=
		\prod_{t=0}^{T-1} \frac{t+1}{t+3}
		=
		\frac{\prod_{t=1}^T t}{\prod_{t=3}^{T+2} t}
		=\frac{2}{(T+1)(T+2)}.
	\end{align*}
\end{proof}

\begin{lemma}\label{eag:lem:3}
	Under conditions of \Cref{eag:thm}, let $\cU^t$ be the following Lyapunov function
	\begin{equation}
		\cU^t = \<a^t+b^t,u^t - u^0> + \frac{\lambda}{2\beta_t}\sqn{a^t + b^t}.
	\end{equation}
	Then, the following inequality holds for all $t \in \{0,1,2,\ldots,T-1\}$:
	\begin{equation}
		\cU^{t+1} \leq (1-\beta_t)\cU^t.
	\end{equation}
\end{lemma}
\begin{proof}
	We start with the monotonicity property of operators $A(u)$ and $B(u)$:
	\begin{equation}\label{eq:5}
		\<u^{t+1} - u^t, (a^{t+1} + b^{t+1}) - (a^t + b^t)> \geq 0,
	\end{equation}
	where $t \in \{0,1,2,\ldots, T-1\}$. From \cref{eag:line:2} of \Cref{alg:eag} we can conclude that
	\begin{equation}\label{eq:6}
		u^{t+1} = u^t + \beta_t(u^0 - u^t) - \lambda (a^{t+1/2} + b^{t+1}),
	\end{equation}
	where $a^{t+1/2} = A(u^{t+1/2})$. From this we also conclude that 
	\begin{equation}\label{eq:7}
		u^{t+1} = u^t + (1-\beta_t)^{-1}\left[\beta_t(u^0 - u^{t+1})- \lambda (a^{t+1/2} + b^{t+1})\right].
	\end{equation}
	Plugging \eqref{eq:6} and \eqref{eq:7} into \eqref{eq:5} gives
	\begin{align*}
		0 &\leq
		(1-\beta_t)^{-1}\<a^{t+1} + b^{t+1},\beta_t (u^0 - u^{t+1}) - \lambda(a^{t+1/2} + b^{t+1})>
		\\&
		-\<a^t+b^t,\beta_t(u^0 - u^t)-\lambda (a^{t+1/2} + b^{t+1})>
		\\&=
		-(1-\beta_t)^{-1}\left[\beta_t \<a^{t+1}+b^{t+1},u^{t+1} - u^0> + \lambda \<a^{t+1} + b^{t+1}, a^{t+1/2} + b^{t+1}>\right]
		\\&
		+\beta_t\<a^t+b^t,u^t - u^0>
		+\lambda\<a^t+b^t,a^{t+1/2} + b^{t+1}>.
	\end{align*}
	Using the parallelogram rule we get
	\begin{align*}
		0 &\leq
		-(1-\beta_t)^{-1}\beta_t \<a^{t+1}+b^{t+1},u^{t+1} - u^0>
		\\&
		-\frac{(1-\beta_t)^{-1}\lambda}{2}\left[\sqn{a^{t+1} + b^{t+1}} + \sqn{a^{t+1/2} + b^{t+1}} - \sqn{a^{t+1} - a^{t+1/2}}\right]
		\\&
		+\beta_t\<a^t+b^t,u^t - u^0>
		\\&
		+\frac{\lambda}{2}\left[\sqn{a^t + b^t} + \sqn{a^{t+1/2} + b^{t+1}} - \sqn{a^t + b^t - (a^{t+1/2} + b^{t+1})}\right]
		\\&=
		\beta_t\<a^t+b^t,u^t - u^0> + \frac{\lambda}{2}\sqn{a^t + b^t}
		\\&
		-(1-\beta_t)^{-1}\left[\beta_t \<a^{t+1}+b^{t+1},u^{t+1} - u^0> + \frac{\lambda}{2}\sqn{a^{t+1} + b^{t+1}}\right]
		\\&
		+\frac{\lambda(1 - (1-\beta_t)^{-1})}{2}\sqn{a^{t+1/2} + b^{t+1}} + \frac{\lambda(1-\beta_t)^{-1}}{2}\sqn{a^{t+1} - a^{t+1/2}}
		\\&
		-\frac{\lambda}{2}\sqn{a^t + b^t - (a^{t+1/2} + b^{t+1})}
		\\&=
		\beta_t\<a^t+b^t,u^t - u^0> + \frac{\lambda}{2}\sqn{a^t + b^t}
		\\&
		-(1-\beta_t)^{-1}\left[\beta_t \<a^{t+1}+b^{t+1},u^{t+1} - u^0> + \frac{\lambda}{2}\sqn{a^{t+1} + b^{t+1}}\right]
		\\&
		-\frac{\lambda\beta_t(1-\beta_t)^{-1}}{2}\sqn{a^{t+1/2} + b^{t+1}} + \frac{\lambda(1-\beta_t)^{-1}}{2}\sqn{a^{t+1} - a^{t+1/2}}
		\\&
		-\frac{\lambda}{2}\sqn{a^t + b^t - (a^{t+1/2} + b^{t+1})}.
	\end{align*}
	Now, we divide both sides of the inequality by $\beta_t$ and get
	\begin{align*}
		0 &\leq
		\<a^t+b^t,u^t - u^0> + \frac{\lambda}{2\beta_t}\sqn{a^t + b^t}
		\\&
		-(1-\beta_t)^{-1}\left[\<a^{t+1}+b^{t+1},u^{t+1} - u^0> + \frac{\lambda}{2\beta_t}\sqn{a^{t+1} + b^{t+1}}\right]
		\\&
		-\frac{\lambda}{2(1-\beta_t)}\sqn{a^{t+1/2} + b^{t+1}} + \frac{\lambda}{2\beta_t(1-\beta_t)}\sqn{a^{t+1} - a^{t+1/2}}
		\\&
		-\frac{\lambda}{2\beta_t}\sqn{a^t + b^t - (a^{t+1/2} + b^{t+1})}.
	\end{align*}
	using the inequality $\sqn{a^{t+1/2} + b^{t+1}} \geq \frac{1}{2}\sqn{a^{t+1} + b^{t+1}} - \sqn{a^{t+1} - a^{t+1/2}}$ we get
	\begin{align*}
		0 &\leq
		\<a^t+b^t,u^t - u^0> + \frac{\lambda}{2\beta_t}\sqn{a^t + b^t}
		\\&
		-(1-\beta_t)^{-1}\left[\<a^{t+1}+b^{t+1},u^{t+1} - u^0> + \frac{\lambda}{2\beta_t}\sqn{a^{t+1} + b^{t+1}}\right]
		\\&
		-\frac{\lambda}{4(1-\beta_t)}\sqn{a^{t+1} + b^{t+1}}
		+\frac{\lambda}{2(1-\beta_t)}\sqn{a^{t+1/2} + a^{t+1}} 
		+ \frac{\lambda}{2\beta_t(1-\beta_t)}\sqn{a^{t+1} - a^{t+1/2}}
		\\&
		-\frac{\lambda}{2\beta_t}\sqn{a^t + b^t - (a^{t+1/2} + b^{t+1})}
		\\&=
		\<a^t+b^t,u^t - u^0> + \frac{\lambda}{2\beta_t}\sqn{a^t + b^t}
		\\&
		-(1-\beta_t)^{-1}\left[\<a^{t+1}+b^{t+1},u^{t+1} - u^0> + \frac{\lambda}{2}\left(\frac{1}{\beta_t} + \frac{1}{2}\right)\sqn{a^{t+1} + b^{t+1}}\right]
		\\&
		+\frac{\lambda}{2\beta_t}\left(
		\frac{1 + \beta_t}{1-\beta_t}\sqn{a^{t+1/2} + a^{t+1}} 
		-\sqn{a^t + b^t - (a^{t+1/2} + b^{t+1})}
		\right).
	\end{align*}
	From the definition of $\beta_t$ it follows that $\frac{1}{\beta_{t+1}} = \frac{1}{\beta_t} + \frac{1}{2}$ and $\frac{1+\beta_t}{1 - \beta_t} \leq 5$. Hence,
	\begin{align*}
		0 &\leq
		\<a^t+b^t,u^t - u^0> + \frac{\lambda}{2\beta_t}\sqn{a^t + b^t}
		\\&
		-(1-\beta_t)^{-1}\left[\<a^{t+1}+b^{t+1},u^{t+1} - u^0> + \frac{\lambda}{2\beta_{t+1}}\sqn{a^{t+1} + b^{t+1}}\right]
		\\&
		+\frac{\lambda}{2\beta_t}\left(
		5\sqn{a^{t+1/2} + a^{t+1}} 
		-\sqn{a^t + b^t - (a^{t+1/2} + b^{t+1})}
		\right).
	\end{align*}
	Using the $M$-Lipschitzness of $A(u)$ we get
	\begin{align*}
		0 &\leq
		\<a^t+b^t,u^t - u^0> + \frac{\lambda}{2\beta_t}\sqn{a^t + b^t}
		\\&
		-(1-\beta_t)^{-1}\left[\<a^{t+1}+b^{t+1},u^{t+1} - u^0> + \frac{\lambda}{2\beta_{t+1}}\sqn{a^{t+1} + b^{t+1}}\right]
		\\&
		+\frac{\lambda}{2\beta_t}\left(
		5M^2\sqn{u^{t+1/2} + u^{t+1}} 
		-\sqn{a^t + b^t - (a^{t+1/2} + b^{t+1})}
		\right).
	\end{align*}
	From \cref{eag:line:1,eag:line:2} of \Cref{alg:eag} it follows that 
	\begin{equation*}
		u^{t+1/2} - u^t = \lambda(a^t + b^t - (a^{t+1/2} + b^{t+1}))
	\end{equation*}
	Hence,
	\begin{align*}
		0 &\leq
		\<a^t+b^t,u^t - u^0> + \frac{\lambda}{2\beta_t}\sqn{a^t + b^t}
		\\&
		-(1-\beta_t)^{-1}\left[\<a^{t+1}+b^{t+1},u^{t+1} - u^0> + \frac{\lambda}{2\beta_{t+1}}\sqn{a^{t+1} + b^{t+1}}\right]
		\\&
		+\frac{\lambda}{2\beta_t}\left(
		5M^2\lambda^2-1
		\right)\sqn{a^t + b^t - (a^{t+1/2} + b^{t+1})}
	\end{align*}
	Using the definition of $\lambda$ we get
	\begin{align*}
		0 &\leq
		\<a^t+b^t,u^t - u^0> + \frac{\lambda}{2\beta_t}\sqn{a^t + b^t}
		\\&
		-(1-\beta_t)^{-1}\left[\<a^{t+1}+b^{t+1},u^{t+1} - u^0> + \frac{\lambda}{2\beta_{t+1}}\sqn{a^{t+1} + b^{t+1}}\right].
	\end{align*}
	Rearranging and multiplying both sides of the inequality by $(1-\beta_t)$ concludes the proof.
	
\end{proof}
Now, we are ready to prove \Cref{eag:thm}.
\begin{proof}[\bf Proof of \Cref{eag:thm}]
	Unrolling the recurrence from \Cref{eag:lem:3} we get
	\begin{equation*}
		\cU^T \leq \prod_{t=0}^{T-1}\cU^0.
	\end{equation*}
	Using \Cref{eag:lem:2} we get
	\begin{equation*}
		\cU^T \leq\frac{2}{(T+1)(T+2)}\cU^0.
	\end{equation*}
	Using the definition of $\cU^t$ we get
	\begin{equation*}
		\frac{\lambda}{\beta_0(T+1)(T+2)} \sqn{a^0 + b^0}\geq \<a^T+b^T,u^T - u^0> + \frac{\lambda}{2\beta_T}\sqn{a^T + b^T}.
	\end{equation*}
	Using the definition of $\beta_t$ we get
	\begin{equation*}
		\frac{3\lambda}{2(T+1)(T+2)} \sqn{a^0 + b^0}\geq \<a^T+b^T,u^T - u^0> + \frac{\lambda(T+3)}{4}\sqn{a^T + b^T}.
	\end{equation*}
	Vector $u^*$ is the solution to problem~\eqref{eq:inclusion}. Hence, there exists $b^* \in B(u^*), a^* = A(u^*)$ such that $a^* + b^* = 0$. From the monotonicity assumption it follows that
	\begin{equation*}
		\<a^T+b^T,u^T - u^*> = 	\<a^T+b^T - (a^* + b^*),u^T - u^*> \geq 0.
	\end{equation*}
	Hence,
	\begin{equation*}
		\frac{3\lambda}{2(T+1)(T+2)} \sqn{a^0 + b^0}\geq \<a^T+b^T,u^* - u^0> + \frac{\lambda(T+3)}{4}\sqn{a^T + b^T}.
	\end{equation*}
	Using the Young's inequality we get
	\begin{align*}
		\frac{3\lambda}{2(T+1)(T+2)} \sqn{a^0 + b^0}&\geq-\frac{\lambda(T+3)}{8}\sqn{a^T+b^T}  - \frac{2}{\lambda(T+3)}\sqn{u^0 - u^*} \\&+ \frac{\lambda(T+3)}{4}\sqn{a^T + b^T}
		\\&=
		\frac{\lambda(T+3)}{8}\sqn{a^T+b^T}  - \frac{2}{\lambda(T+3)}\sqn{u^0 - u^*}.
	\end{align*}
	After rearranging we get
	\begin{align*}
		\frac{\lambda(T+3)}{8}\sqn{a^T+b^T} \leq
		\frac{2}{\lambda(T+3)}\sqn{u^0 - u^*} + \frac{3\lambda}{2(T+1)(T+2)} \sqn{a^0 + b^0}.
	\end{align*}
	Multiplying both sides of the inequality by $\frac{8}{\lambda(T+3)}$ gives
	\begin{align*}
		\sqn{a^T+b^T} \leq
		\frac{8}{\lambda^2(T+3)^2}\sqn{u^0 - u^*} + \frac{12\lambda}{(T+1)(T+2)(T+3)} \sqn{a^0 + b^0}.
	\end{align*}
	Plugging the definition of $\lambda$ gives
	\begin{align*}
		\sqn{a^T+b^T} &\leq
		\frac{40M^2}{(T+3)^2}\sqn{u^0 - u^*} + \frac{12}{(T+1)(T+2)(T+3)} \sqn{a^0 + b^0}
		\\&\leq
		\frac{1}{(T+1)^2}\left(96 M^2 \sqn{u^0 - u^*} + 12\sqn{a^0 + b^0}\right)
	\end{align*}
	Using \Cref{eag:lem:1} we get
	\begin{align*}
		\sqn{a^T+b^T} &\leq
		\frac{288M^2}{(T+1)^2}\sqn{u^{-1} - u^*}.
	\end{align*}
\end{proof}

\newpage 

\section{Proof of \Cref{lem:M}}
Monotonicity of $A^k(u)$ can be verified trivially. Maximal monotonicity of $B(u)$ follows from the fact that it is equal to the subdifferential of a convex function $(\gamma_x^{-1/2} x, \gamma_y^{-1/2}y)\mapsto r(x) + g(y)$ and \citep[Theorem 12.17]{rockafellar2009variational}.

Now, we prove Lipschitzness  of operator $A^k(u)$. Denoting $u_1=(\gamma_x^{-1/2}x_1, \gamma_y^{-1/2}y_1)$ and $u_2=(\gamma_x^{-1/2}x_2, \gamma_y^{-1/2}y_2)$ gives 
\begin{align*}
	\sqn{A^k(u_1) - A^k(u_2)}
	&\leq
	\gamma_x\sqn{a_x^k(x_1,y_1)-a_x^k(x_2,y_2)}
	+\gamma_y\sqn{a_y^k(x_1,y_1) - a_y^k(x_2,y_2)}
	\\&\leq
	2\gamma_x\sqn{a_x^k(x_1,y_1)-a_x^k(x_1,y_2)}
	+2\gamma_x\sqn{a_x^k(x_1,y_2)-a_x^k(x_2,y_2)}
	\\&
	+2\gamma_y\sqn{a_y^k(x_1,y_1) - a_y^k(x_1,y_2)}
	+2\gamma_y\sqn{a_y^k(x_1,y_2) - a_y^k(x_2,y_2)}
\end{align*}
Using the definition of operators $a_x^k(x,y)$ and $a_y^k(x,y)$, the definition of function $\F(x,y)$ and  \Cref{ass:smoothness,ass:convexity,ass:concavity} we get
\begin{align*}
	\sqn{A^k(u_1) - A^k(u_2)}
	&\leq
	2\gamma_xL^2\sqn{y_1 - y_2}
	+2\gamma_xL^2\sqn{x_1 - x_2}
	\\&
	+2\gamma_y (L+\theta_y^{-1})^2\sqn{y_1 - y_2}
	+2\gamma_y L^2\sqn{x_1 - x_2}
	\\&=
	2(\gamma_x\gamma_yL^2 + \gamma_y^2 (L+\theta_y^{-1})^2)\gamma_y^{-1}\sqn{y_1 - y_2}
	\\&
	+2(\gamma_x^2 L^2 + \gamma_x\gamma_y L^2)\gamma_x^{-1}\sqn{x_1 - x_2}
	\\&\leq
	4\max\{\gamma_x^2L^2, \gamma_x\gamma_yL^2,\gamma_y^2 (L+\theta_y^{-1})^2\}\sqn{u_1 - u_2}
	\\&=
	4\max\{\gamma_x^2L^2,\gamma_y^2 (L+\theta_y^{-1})^2\}\sqn{u_1 - u_2}.
\end{align*}
\qed

\section{Proof of \Cref{lem:T}}
To prove this lemma it is sufficient to show that condition~\eqref{foam:while} is satisfied when $t = T$, where $T$ is given by \eqref{foam:T}. 

Let $(x^{k,*},y^{k,*})$ be the solution of the monotone inclusion problem
\begin{equation*}
	0 \in A^k((\gamma_x^{-1/2}x^{k,*}, \gamma_y^{-1/2}y^{k,*})) + B((\gamma_x^{-1/2}x^{k,*} \gamma_y^{-1/2}y^{k,*})).
\end{equation*}
Note, that this solution always exists. One can easily show that operator $A^k(u)$ is strongly monotone, which implies the following inequality:
\begin{align*}
	\<x^{k,t} - x^{k,*},a_x^k(x^{k,t}, y^{k,t}) + b_x^{k,t}> &+ \<y^{k,t} - y^{k,*},a_y^k(x^{k,t}, y^{k,t}) + b_y^{k,t}>\geq
	\\
	&\geq \frac{\mu_x}{2}\sqn{x^{k,t} - x^{k,*}} + \theta_y^{-1}\sqn{y^{k,t} - y^{k,*}}
	\\&\geq \gamma_x^{-1}\sqn{x^{k,t} - x^{k,*}} + \gamma_y^{-1}\sqn{y^{k,t} - y^{k,*}}.
\end{align*}
The latter inequality implies
\begin{align*}
	\gamma_x^{-1}\sqn{x^{k,t} - x^{k,*}} &+ \gamma_y^{-1}\sqn{y^{k,t} - y^{k,*}}\leq\\
	&\leq
	\gamma_x\sqn{a_x^k(x^{k,t}, y^{k,t}) + b_x^{k,t}}
	+\gamma_y\sqn{a_y^k(x^{k,t}, y^{k,t}) + b_y^{k,t}}.
\end{align*}
Further, we get
\begin{align*}
	\gamma_x^{-1}\sqn{x^{k,-1} - x^{k,*}} &+ \gamma_y^{-1}\sqn{y^{k,-1} - y^{k,*}}
	\\&\leq
	2\gamma_x^{-1}\sqn{x^{k,-1} - x^{k,t}} + 2\gamma_y^{-1}\sqn{y^{k,-1} - y^{k,t}}
	\\&
	+2\gamma_x^{-1}\sqn{x^{k,t} - x^{k,t}} + 2\gamma_y^{-1}\sqn{y^{k,t} - y^{k,t}}
	\\&\leq
	2\gamma_x^{-1}\sqn{x^{k,-1} - x^{k,t}} + 2\gamma_y^{-1}\sqn{y^{k,-1} - y^{k,t}}
	\\&
	+2\gamma_x\sqn{a_x^k(x^{k,t}, y^{k,t}) + b_x^{k,t}}
	+2\gamma_y\sqn{a_y^k(x^{k,t}, y^{k,t}) + b_y^{k,t}}.
\end{align*}

Using \Cref{eag:thm} and the definition~\eqref{eq:AB} of operators $A^k(u)$ and $B(u)$ we get the following inequality:
\begin{align*}
	\gamma_x\sqn{a_x^k(x^{k,t}, y^{k,t}) + b_x^{k,t}} &+ \gamma_y\sqn{a_y^k(x^{k,t}, y^{k,t}) + b_y^{k,t}} 
	\\& \leq
	\frac{288M^2}{(t+1)^2}\left(\gamma_x^{-1}\sqn{x^{k,-1} - x^{k,*}} + \gamma_y^{-1}\sqn{y^{k,-1} - y^{k,*}}\right)
	\\&\leq
	\frac{288M^2}{(t+1)^2}\left(2\gamma_x^{-1}\sqn{x^{k,-1} - x^{k,t}} + 2\gamma_y^{-1}\sqn{y^{k,-1} - y^{k,t}}\right)
	\\&
	+\frac{288M^2}{(t+1)^2}\left(2\gamma_x\sqn{a_x^k(x^{k,t}, y^{k,t}) + b_x^{k,t}}
	+2\gamma_y\sqn{a_y^k(x^{k,t}, y^{k,t}) + b_y^{k,t}}\right)
	\\&=
	\frac{576M^2}{(t+1)^2}\left(\gamma_x^{-1}\sqn{x^{k,-1} - x^{k,t}} + \gamma_y^{-1}\sqn{y^{k,-1} - y^{k,t}}\right)
	\\&
	+\frac{576M^2}{(t+1)^2}\left(\gamma_x\sqn{a_x^k(x^{k,t}, y^{k,t}) + b_x^{k,t}}
	+s\gamma_y\sqn{a_y^k(x^{k,t}, y^{k,t}) + b_y^{k,t}}\right)
\end{align*}
Now, we set $t = T$, where $T$ is defined by \eqref{foam:T}, and use the definition~\eqref{foam:M} of $M$. This implies.
\begin{align*}
	\gamma_x\sqn{a_x^k(x^{k,t}, y^{k,t}) + b_x^{k,t}} &+ \gamma_y\sqn{a_y^k(x^{k,t}, y^{k,t}) + b_y^{k,t}} 
	\\&\leq
	\frac{1}{2}\left(\gamma_x^{-1}\sqn{x^{k,-1} - x^{k,t}} + \gamma_y^{-1}\sqn{y^{k,-1} - y^{k,t}}\right)
	\\&
	+\frac{1}{2}\left(\gamma_x\sqn{a_x^k(x^{k,t}, y^{k,t}) + b_x^{k,t}}
	+\gamma_y\sqn{a_y^k(x^{k,t}, y^{k,t}) + b_y^{k,t}}\right).
\end{align*}
After rearranging, we obtain condition~\eqref{foam:while} with $t = T$.

Now, one can observe that there are two possibilities: the while-loop of \Cref{alg:foam} stops when $t=T$, otherwise it stops earlier. This concludes the proof.\qed

\section{Proof of \Cref{foam:thm} }
According to \Cref{pp:thm}, the following number of outer iterations of \Cref{alg:foam} are required to find an $\epsilon$-accurate solution:
\begin{equation}
	K = \cO\left(\max\left\{\frac{1}{\alpha}, \frac{\alpha}{\theta_y\mu_y}\right\}\log\frac{1}{\epsilon}\right).
\end{equation}
Using the definition of $\alpha$ and $\theta_y$ we get
\begin{equation}
	K = \cO\left(\max\left\{1, \sqrt{\frac{\mu_x}{\mu_y}}\right\}\log\frac{1}{\epsilon}\right).
\end{equation}
According to \Cref{lem:T} at most $T$ inner iterations are performed by \Cref{alg:foam} at each outer iteration, where $T$ is given as
\begin{equation}
	T = \cO \left(\max \left\{\frac{L}{\mu_x}, \theta_y L \right\} \right).
\end{equation}
Using the definition of $\theta_y$ we get
\begin{equation}
	T = \cO \left(\frac{L}{\mu_x} \right).
\end{equation}
Now, we get the total number of iterations:
\begin{equation}
	K \times T = \cO\left(\max\left\{1, \sqrt{\frac{\mu_x}{\mu_y}}\right\}\log\frac{1}{\epsilon}\right) \times \cO \left(\frac{L}{\mu_x} \right)
	=\cO\left(\max\left\{\frac{L}{\mu_x}, \frac{L}{\sqrt{\mu_x\mu_y}}\right\}\log \frac{1}{\epsilon}\right).
\end{equation}
It remains to observe that \Cref{alg:foam} performs $\cO(1)$ gradient evaluations per iteration.
\qed
\end{document}